\documentclass[11pt,reqno,final]{article}
\usepackage{amsthm}
\usepackage{amsmath,amsfonts,amssymb,amsthm,version}
\usepackage{mathrsfs,fancybox,pifont}
\usepackage{graphicx}
\usepackage{url,hyperref}
\usepackage[notcite,notref]{showkeys}
\usepackage{color}
\usepackage{subfigure,multirow}
\usepackage{epstopdf}
\usepackage{cases}
\usepackage{mathtools}
\usepackage{algorithm,algorithmic}
\usepackage{authblk}
\usepackage{fancyhdr}
\usepackage{lipsum}
\usepackage{cite}
\usepackage{float}
\usepackage[top=2cm]{geometry}

\allowdisplaybreaks

\numberwithin{equation}{section}
\numberwithin{table}{section}

\theoremstyle{plain}
\newtheorem{theorem}{Theorem}[section]
\newtheorem{lemma}{Lemma}[section]
\newtheorem{corollary}{Corollary}[section]

\theoremstyle{definition}

\theoremstyle{remark}
\newtheorem{remark}{Remark}[section]

\title{Bifurcation Analysis of a Predator–Prey Model with Allee Effect and Cooperative Hunting}
\author{Yujie Gao }  
\author{Ton Viet Ta}  
\affil{Mathematical Modeling Laboratory, Kyushu University\\
744 Motooka, Nishi Ward, Fukuoka 819-0395, Japan}
\date{\today}

\begin{document}
\maketitle

\begin{abstract}
We propose a novel predator–prey model that integrate two ecologically significant mechanisms: the Allee effect in the prey population and cooperative hunting behavior among predators. Building upon the Rosenzweig–MacArthur framework, our model modifies the prey growth term to incorporate the Allee effect and introduces a nonlinear functional response reflecting predator cooperation. 

We establish the existence and boundedness of global solutions for the system and analyze the local and global stability of its equilibria. In addition, we perform a comprehensive bifurcation analysis, including transcritical, saddle-node, Hopf, and heteroclinic bifurcations, to explore how system dynamics change with key parameters. These results reveal rich and biologically relevant behaviors, such as multiple equilibria, transitions in stability, and the emergence of complex dynamical patterns.

\medskip
\noindent{\bf Keywords}: Rosenzweig–MacArthur model, Allee effect, cooperative hunting, predator–prey dynamics, bifurcation analysis
\end{abstract}

\section{Introduction}

Predator–prey interactions are a cornerstone of ecological dynamics. Over the past several decades, many mathematical models have been developed to describe and analyze these interactions. A foundational framework in this field is the Lotka–Volterra model, introduced independently by Lotka and Volterra in the 1920s~\cite{lotka1925elements, lotka1927fluctuations}. This pioneering model presents a simplified view of predator–prey systems, assuming logistic growth for the prey and a specialist predator wholly dependent on the prey for survival. The model also assumes a constant mortality rate for the predator.

Building upon this, Rosenzweig and MacArthur proposed a more biologically realistic model by incorporating density-dependent prey growth and a Holling type II functional response~\cite{rosenzweig1963graphical}. Their model is formulated as the following system of differential equations:
\begin{equation}\label{equ01}
\left\{
\begin{aligned}
&\frac{dx(t)}{dt} = x(t)f(x(t)) - \frac{\lambda x(t)y(t)}{1 + h \lambda x(t)}, \\
&\frac{dy(t)}{dt} = \frac{\lambda x(t)y(t)}{1 + h \lambda x(t)} - s y(t),
\end{aligned}
\right.
\end{equation}
where \( x(t) \) and \( y(t) \) denote the densities of prey and predator populations at time \( t \), respectively. The prey's growth rate is governed by \( x(t)f(x(t)) \), where \( f(x) \) is a bounded function that becomes negative when the prey exceeds its carrying capacity \( K > 0 \). The term \( \frac{\lambda x(t)}{1 + h\lambda x(t)} \) models the Holling type II functional response~\cite{holling1959some}, with \( \lambda \) as the prey encounter rate and \( h \) as the predator’s handling time. The constant \( s \) represents the predator’s per-capita death rate.

The Rosenzweig–MacArthur model has served as the basis for extensive investigations into predator–prey dynamics. Dercole et al.~\cite{dercole2003bifurcation} analyzed local and codimension-two bifurcations; Sugie and Saito~\cite{sugie2012uniqueness} focused on the uniqueness of limit cycles via Hopf bifurcation; Yuan and Wang~\cite{yuan2023bifurcation} extended the model to a stochastic setting; and Lu~\cite{lu2021global} examined both local and global bifurcations in the Bazykin variant.

However, the classical model does not capture the \emph{Allee effect}, an ecologically important phenomenon in prey dynamics. The Allee effect describes a positive correlation between population density and individual fitness~\cite{kramer2009evidence, drake2004allee}. At low densities, prey populations may suffer reduced growth due to difficulties in mate finding, diminished cooperative defense, or limited resource acquisition~\cite{HARTONO2024109112,Ta_2018}. As density increases, these constraints lessen, enhancing growth rates up to the carrying capacity.

Mathematically, the Allee effect is often modeled via the prey growth term:
\[
r_1 x(t) \left(1 - \frac{x(t)}{k_1}\right)(x(t) - k_0),
\]
where \( r_1 > 0 \) is the intrinsic growth rate, \( k_1 > k_0 \) is the carrying capacity, and \( k_0 \) is the Allee threshold. A positive \( k_0 \) indicates a \emph{strong} Allee effect, while a negative \( k_0 \) corresponds to a \emph{weak} Allee effect~\cite{berec2007multiple, begon2014essentials}.

Figure~\ref{figure9} illustrates the contrast between logistic growth and prey growth with weak or strong Allee effects.
\begin{figure}[H]
  \begin{center}
    \includegraphics[scale=0.35]{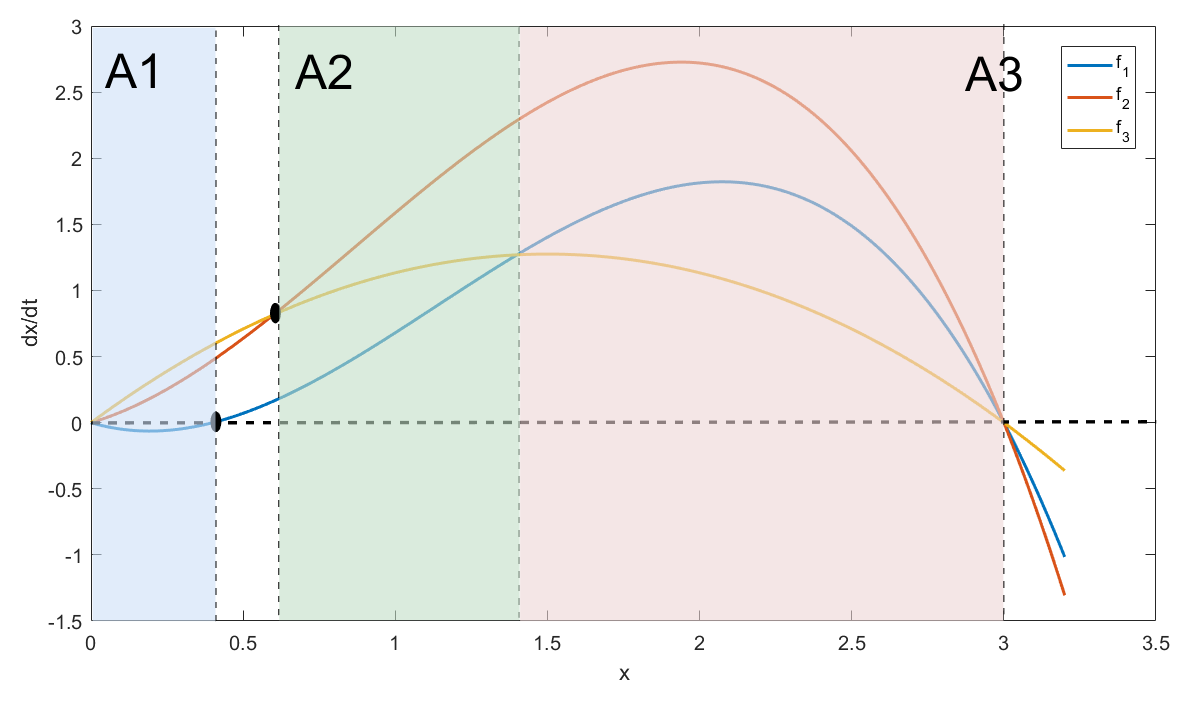}
    \caption{Comparison of prey growth functions: strong Allee effect (blue), weak Allee effect (red), and logistic growth (yellow). The strong Allee effect results in negative growth in region \( A_1 \). For the same intrinsic growth rate \( r_1 \), the weak Allee effect consistently shows higher growth than the strong Allee effect. In regions \( A_2 \) and \( A_3 \), the weak Allee effect surpasses logistic growth, and in region \( A_3 \), the strong Allee effect also exceeds logistic growth.}
    \label{figure9}
  \end{center}
\end{figure}

A growing body of literature has explored the role of the Allee effect in predator–prey models. Sun~\cite{sun2016mathematical} developed a general framework for its inclusion. Morozov et al.~\cite{morozov2004bifurcations} studied bifurcations and chaotic dynamics, while Terry~\cite{terry2015predator} explored its influence on predator reproduction. Other works have addressed stability and bifurcation behavior in such models~\cite{vishwakarma2021role, aguirre2013stochastic}.

In response to reduced prey encounter rates due to prey aggregation, many predators have evolved cooperative hunting strategies. This behavior increases foraging efficiency, allows targeting of larger prey, and reduces predation risk~\cite{sahoo2021impact, sahoo2023modeling}. To mathematically incorporate this, Alves and Hilker~\cite{alves2017hunting} proposed a linear functional response \( g(y) = \lambda + a y \). Later, Dey et al.~\cite{dey2022bifurcation} introduced a more refined form with saturating cooperation:
\begin{equation}\label{lambda(y)}
\lambda(y) = \lambda_0 + \frac{a y}{b + y},
\end{equation}
where \( \lambda_0 \) is the base encounter rate, \( a \) is the maximum cooperative gain, and \( b \) is the half-saturation constant.

In this paper, we propose a novel predator–prey model that integrates both the Allee effect and cooperative hunting \eqref{lambda(y)}:
\begin{equation}\label{equ02}
\begin{cases}
\begin{aligned}
\frac{dx(t)}{dt} &= r_1 x(t) \left(1 - \frac{x(t)}{k_1} \right)(x(t) - k_0) - \frac{(\lambda + A y(t)) x(t) y(t)}{b + y(t) + h x(t)(\lambda + A y(t))}, \\
\frac{dy(t)}{dt} &= \frac{(\lambda + A y(t)) x(t) y(t)}{b + y(t) + h x(t)(\lambda + A y(t))} - s y(t),
\end{aligned}
\end{cases}
\end{equation}
with initial conditions \( (x(0), y(0)) \geq 0 \). Here, \( r_1 > 0 \), \( k_0 \in \mathbb{R} \), and \( k_1 > \max\{k_0, 0\} \) are prey-specific parameters. The parameters \( \lambda = \lambda_0 b > 0 \) and \( A = \lambda_0 + a > 0 \) describe cooperative hunting. The remaining constants \( b, s, h > 0 \) retain their biological interpretations from~\eqref{equ01}. The model exhibits a weak Allee effect when \( k_0 < 0 \) and a strong Allee effect when \( k_0 > 0 \).

We focus on analyzing the dynamical behavior of this model. We first establish the existence and boundedness of global solutions and characterize the system’s equilibria. We then investigate their local and global stability and conduct a detailed bifurcation analysis—covering transcritical, saddle-node, Hopf, and heteroclinic bifurcations. These bifurcations elucidate how varying parameters such as \( \lambda \) impact system behavior. For instance, near a transcritical bifurcation, the boundary equilibrium may change from a saddle to a stable node while a new saddle appears nearby, indicating that small shifts in parameters can induce significant qualitative changes in long-term population dynamics.

The structure of the paper is as follows. In Section~\ref{sec2}, we prove the existence and boundedness of global solutions and discuss the existence of interior equilibria. Section~\ref{sec3} focuses on the stability of equilibria, including global stability conditions. Section~\ref{sec4} presents a comprehensive bifurcation analysis, covering transcritical, saddle-node, Hopf, and heteroclinic bifurcations. Finally, Section~\ref{sec9} provides concluding remarks.

\section{Existence of global solutions and equilibria}\label{sec2}

In this section, we prove the existence and boundedness of global positive solutions for system \eqref{equ02} and investigates conditions for the existence of interior equilibria. First, we show the existence and uniqueness of global positive solutions in this mathematical model.

\begin{theorem}
For any non-negative initial values $(x_0,y_0) \in \overline {\mathbb{R}^2_+}$, where $ \mathbb{R}^2_+=\{(x,y)\mid x, y> 0\}$, there exists a unique global positive solution $(x(t), y(t))$ to system \eqref{equ02} starting from  $(x(0),y(0))=(x_0,y_0)$. 
\end{theorem}

\begin{proof}
Since the right-hand side of system \eqref{equ02} is locally Lipschitz continuous in a neighborhood of $(x_0, y_0)$, there exists a unique local solution $(x(t), y(t))$ on the maximal interval $[0, \tau)$, where either $\tau = \infty$ or $\tau > 0$ is an explosion time. That is,  
\[
\lim_{t \to \tau} (|x(t)| + |y(t)|) = \infty \quad \text{or} \quad \lim_{t \to \tau} [b + y(t) + hx(t)(\lambda + Ay(t))] = 0.
\]  
We consider two cases.  

\textbf{Case 1:} $x_0 = 0$ or $y_0 = 0$.  

By uniqueness, if $x_0 = 0$, then $x(t) = 0$ for all $t \in [0, \tau)$, and similarly, if $y_0 = 0$, then $y(t) = 0$ for all $t \in [0, \tau)$. In these cases, system \eqref{equ02} reduces to a single equation, which has a global non-negative solution. Therefore, $\tau = \infty$.  

\textbf{Case 2:} $x_0, y_0 > 0$.  

Define the times 
\[
\tau_x = \inf \{ t \in [0, \tau) \mid x(t) = 0 \}, \quad \tau_y = \inf \{ t \in [0, \tau) \mid y(t) = 0 \},
\]  
with the convention that $\inf \varnothing = \tau$.  

We claim that $\tau_x = \tau_y = \tau$. Suppose, for contradiction, that $\tau > \tau_x$. Consider the Cauchy problem:  
\[
\begin{cases}
\frac{d x(t)}{d t} = r_1 x(t) \left( 1 - \frac{x(t)}{k_1} \right) (x(t) - k_0) - \frac{(\lambda + A y(t)) x(t) y(t)}{b + y(t) + h x(t)(\lambda + A y(t))}, \\
x(\tau_x) = 0.
\end{cases}
\]  
By uniqueness, the solution must satisfy $x(t) \equiv 0$ for all $t$ in $[\tau_x, \tau)$ and $[0,\tau_x)$. However, this contradicts our assumption that $x(0) > 0$. Thus, $\tau = \tau_x$. A similar argument shows $\tau = \tau_y$.  

Next, we show that $\tau = \infty$. Suppose, for contradiction, that $\tau$ is finite. From the definition of $\tau$, $\tau_x$, and $\tau_y$, we have
\begin{equation}  \label{tau_finite}
    \lim_{t \to \tau} (|x(t)| + |y(t)|) = \infty.
\end{equation}
Define the Lyapunov function  
\begin{equation*}\label{equ05}
    V(x, y) = \frac{1}{2} (x^2 + y^2).
\end{equation*}
Using system \eqref{equ02}, we obtain  
\begin{align*}
\frac{dV(x(t), y(t))}{dt} &= x\frac{d x}{dt} + y\frac{dy}{dt}\\
&\leq r_1 x^2 \left( 1 - \frac{x}{k_1} \right) (x - k_0) + y \left[ \frac{(\lambda + A y) x y}{b + y + h x (\lambda + A y)} - s y \right] \\
&\leq r_1 x^2 \left( 1 - \frac{x}{k_1} \right) (x - k_0) + \frac{(\lambda + A y) x y^2}{b + y + h x (\lambda + A y)}.
\end{align*}  
Define constants  
\[
L_1 = \max_{x > 0} r_1 \left( 1 - \frac{x}{k_1} \right) (x - k_0), \quad
L_2 = \max_{x, y > 0} \frac{\lambda + A x}{b + y + h x (\lambda + A y)}.
\]  
Then, for every $t \in [0, \tau]$,  
\[
\frac{dV(x(t), y(t))}{dt} \leq L_1 x^2 + L_2 y^2 \leq \max \{ L_1, L_2 \} (x^2 + y^2) = 2 \max \{ L_1, L_2 \} V(x, y).
\]  

By Gronwall’s inequality,  
\[
V(x(t), y(t)) \leq V(x_0, y_0) e^{2 \max \{ L_1, L_2 \} \tau} < \infty, \quad \text{for all } t \in [0, \tau].
\]  
However, from \eqref{tau_finite}, we have   
\[
\lim_{t \to \tau} V(x(t), y(t)) = \infty.
\]  
This contradiction implies that $\tau = \infty$, completing the proof.  
\end{proof}

Second, we prove the boundedness of solutions.
\begin{theorem}\label{thm02}
    Let \((x(t),y(t))\) be the solution of system \eqref{equ02} with initial condition \((x_0,y_0) \in \mathbb{R}_+^2\). Then, the solution remains bounded for all \( t \in [0,\infty) \).
\end{theorem}
\begin{proof}
Consider the auxiliary system:
\begin{equation}\label{equ06}
\left\{\begin{aligned}
&\frac{d \bar{x}(t)}{d t}=r_1\bar{x}(t)\big(1-\frac{\bar{x}(t)}{k_1}\big)\big(\bar{x}(t)-k_0 \big), \\
& \bar{x}(0) = x_0.
\end{aligned}\right.
\end{equation}
Since \( x(t) \) and \( y(t) \) remain positive, applying the comparison theorem yields \( x(t) \leq \bar{x}(t) \) for all \( t \in [0, \infty) \). Moreover, it is evident that \( \bar{x}(t) \leq \max\{k_1, x_0\} \), leading to the bound:
$$x(t) \leq \max\{k_1, x_0\}.$$

Define \( S(t) = x(t) + y(t) \), so that \( y(t) \leq S(t) \). From system \eqref{equ02}, we analyze the strong Allee effect case (i.e. $k_0>0$):
\begin{align*}
    \frac{d S(t)}{d t} &= r_1 x(t) \left( 1 - \frac{x(t)}{k_1} \right) \left( x(t) - k_0 \right) - s y(t) \\
    &= r_1 x^2(t) - \frac{r_1}{k_1} x^3(t) - r_1 k_0 x(t) + \frac{r_1}{k_1} k_0 x^2(t) - s y(t) \\
    &\leq \left( r_1 + \frac{r_1}{k_1} k_0 \right) {\max}^2\{k_1, x_0\} - \min\{r_1 k_0, s\} S(t).
\end{align*}
Applying Gronwall’s inequality confirms the boundedness of $S(t)$ (and therefore \( y(t) \)) under the strong Allee effect.

For the weak Allee effect case (i.e. $k_0<0$), we have:
\begin{align*}
    \frac{d S(t)}{d t} &= r_1 x^2(t) - \frac{r_1}{k_1} x^3(t) - r_1 k_0 x(t) + \frac{r_1}{k_1} k_0 x^2(t) - s y(t) \\
    &\leq r_1 x^2(t) - r_1 k_0 x(t) - s x(t) + s x(t) + \frac{r_1}{k_1} k_0 x^2(t) - s y(t) \\
    &\leq r_1 {\max}^2\{k_1, x_0\} + (s - r_1 k_0) \max\{k_1, x_0\} - s S(t).
\end{align*}
Thus, \( S(t) \) remains bounded, implying that \( y(t) \) is also bounded. The proof is complete.
\end{proof}

In Theorem \ref{thm02}, we established the boundedness of solutions depending on the initial conditions. The following theorem strengthens this result by showing that, after a sufficiently large time, the solutions remain bounded by constants that are independent of the initial values.

For $\delta\geq 0$, define the set 
\begin{equation} \label{B_delta}
{\mathbf B^\delta}=(0, k_1+\delta] \times (0, k_1+\delta].
\end{equation}
\begin{theorem}\label{lem08}
Let \( (x(t), y(t)) \) be the solution to system~\eqref{equ02} with any initial condition \( (x_0, y_0) \in \mathbb{R}_+^2 \). Then,
\[
\left(\limsup_{t \to \infty} x(t), \limsup_{t \to \infty} y(t)\right) \in \mathbf{B^0}.
\]

In addition, \( \mathbf{B^\delta}\, (\delta>0) \) is a positively invariant set to system~\eqref{equ02}.
\end{theorem}

\begin{proof}
Similarly to the proof of Theorem~\ref{thm02}, define \( S(t) = x(t) + y(t) \), and observe that
\[
S(t) \leq \bar{x}(t), \quad \text{for all } t \geq 0,
\]
where \( \bar{x}(t) \) is the solution to the differential equation:
\begin{equation}\label{equ04}
\left\{
\begin{aligned}
&\frac{d \bar{x}(t)}{d t} = r_1 \bar{x}(t)\left(1 - \frac{\bar{x}(t)}{k_1}\right)\left(\bar{x}(t) - k_0 \right), \\
&\bar{x}(0) = x(0) + y(0).
\end{aligned}
\right.
\end{equation}

It is clear that
\[
\bar{x}(t) \leq \max\{k_1, \bar{x}(0)\}, \quad \text{for all } t \geq 0.
\]
Therefore, if the initial value satisfies \( \bar{x}(0) = x(0) + y(0) \leq k_1 \), the conclusion of the theorem immediately holds.

Now consider the case \( \bar{x}(0) > k_1 \). Since
\[
\frac{d \bar{x}(t)}{d t} = r_1 \bar{x}(t)\left(1 - \frac{\bar{x}(t)}{k_1}\right)\left(\bar{x}(t) - k_0 \right) < 0, \quad \text{whenever } \bar{x}(t) > k_1,
\]
and \( \bar{x}(t) = k_1 \) is an equilibrium of equation~\eqref{equ04}, it follows that
\[
\bar{x}(t) \in (k_1, \bar{x}(0)], \quad \text{for all } t \geq 0,
\]
and \( \bar{x}(t) \) is a decreasing function of \( t \). This implies that \( \mathbf{B^\delta} \) is a positively invariant set to system~\eqref{equ02}.

Let \( \ell = \lim_{t \to \infty} \bar{x}(t) \in [k_1, \bar{x}(0)] \). We claim that \( \ell = k_1 \). Suppose, for contradiction, that \( \ell > k_1 \). Then
\[
\lim_{t \to \infty} \frac{d \bar{x}(t)}{d t} = -\varepsilon,
\]
where \( \varepsilon = r_1 \ell \left( \frac{\ell}{k_1} - 1 \right)(\ell - k_0) > 0 \).
Hence, there exists \( t_0 > 0 \) such that
\[
\frac{d \bar{x}(t)}{d t} \leq -\frac{\varepsilon}{2}, \quad \text{for all } t \geq t_0.
\]
This implies
\[
\int_{t_0}^{\infty} \frac{d \bar{x}(t)}{d t} \, dt \leq -\int_{t_0}^{\infty} \frac{\varepsilon}{2} \, dt = -\infty,
\]
which contradicts the existence of the finite limit \( \ell \). Therefore, our assumption is false and it must be that \( \ell = k_1 \).

Since \( S(t) \leq \bar{x}(t) \), it follows that
\[
\limsup_{t \to \infty} S(t) \leq \lim_{t\to\infty} \bar{x}(t) = k_1.
\]
  Hence,
\[
\left( \limsup_{t \to \infty} x(t), \limsup_{t \to \infty} y(t) \right) \in \mathbf{B^0}.
\]
This completes the proof.
\end{proof}

Finally, we analyze the existence and number of interior equilibria for system \eqref{equ02} by solving the following algebraic system:
\begin{equation} \label{equ08}
\begin{cases}
f^{(1)}(x,y) = 0, \\
f^{(2)}(x,y) = 0,
\end{cases}
\end{equation}
where $f^{(1)}(x,y) $ and $f^{(2)}(x,y)$ are defined as 
\begin{equation}\label{equ07}
\left\{\begin{aligned}
f^{(1)}(x,y) &= r_1 \left(1 - \frac{x}{k_1} \right) (x - k_0) - \frac{(\lambda + A y)y}{b + y + h x (\lambda + A y)}, \\
f^{(2)}(x,y) &= \frac{(\lambda + A y)x}{b + y + h x (\lambda + A y)} - s.
\end{aligned}\right.
\end{equation}

From the second equation in \eqref{equ08}, we solve for \( y \):
\begin{equation} \label{TNB1}
\frac{\lambda + A y}{b + y + h x (\lambda + A y)} = \frac{s}{x},
\end{equation}
which simplifies to
\begin{align}\label{equ22}
    y = g^{(2)}(x), \quad \text{where} \quad g^{(2)}(x) = \frac{(\lambda - s h \lambda)x - s b}{s + A(sh - 1)x}= \frac{\lambda}{A}\left[-1 + \frac{s(1-\frac{A}{\lambda}b)}{A(sh-1)x+s}\right].
\end{align}

Substituting \eqref{TNB1} into the first equation of \eqref{equ08} gives
\begin{align}\label{equ21}
     y = g^{(1)}(x), \quad \text{where} \quad g^{(1)}(x) = \frac{r_1}{s} x \left(1 - \frac{x}{k_1} \right) (x - k_0).
\end{align}
Thus, the interior equilibria of system \eqref{equ02} correspond to the intersection points of the curves \( y = g^{(1)}(x) \) and \( y = g^{(2)}(x) \) in the positive quadrant \( \mathbb{R}_+^2 \). Figure \ref{figure1} presents the graphical representations of these functions.

\begin{figure}[H]
  \begin{center} 
    \includegraphics[scale=0.35]{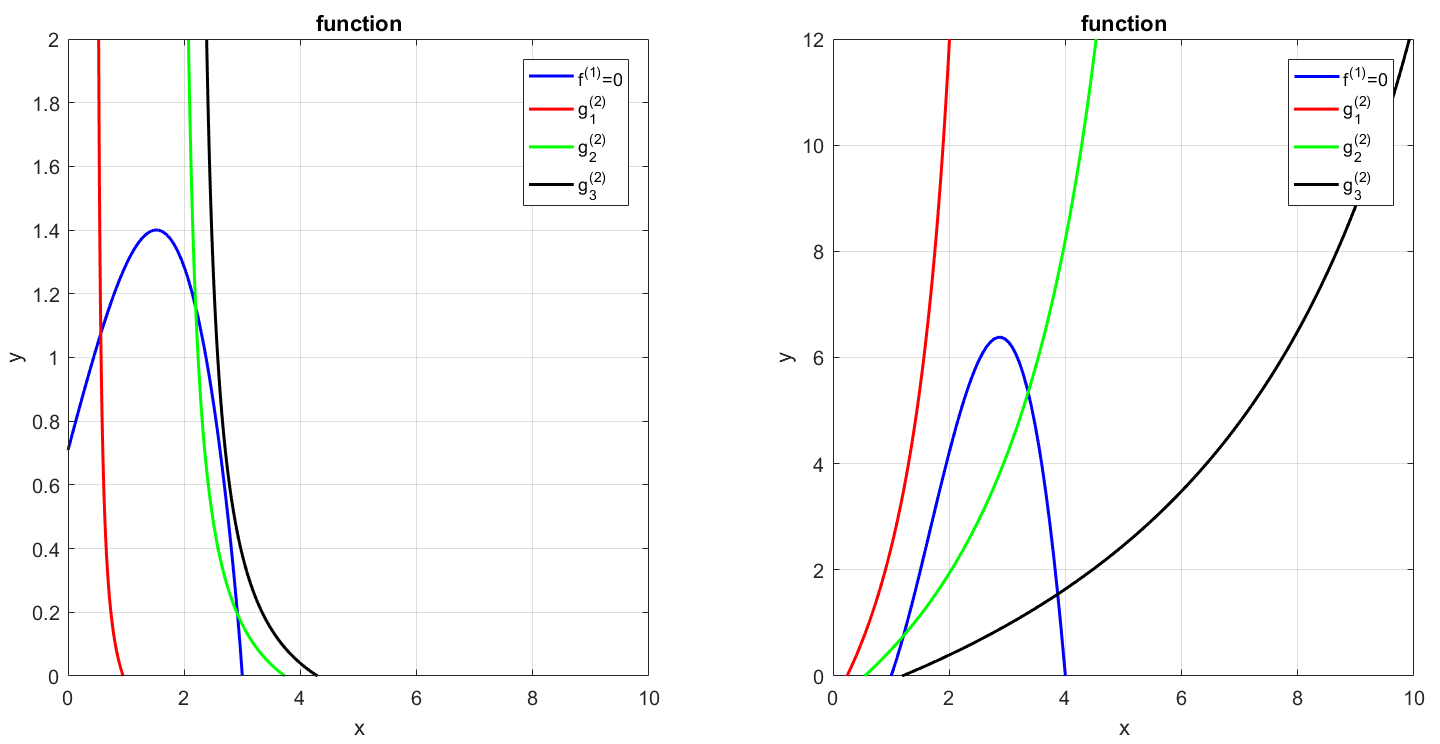}
    \caption{Graphs of $f^{(1)}(x,y)=0$ (blue curve, independent of $s$) and $g^{(2)}(x)$, together with their intersection points. 
Left panel (weak Allee effect): $r_1=0.45$, $k_1=3$, $k_0=-1$, $\lambda=0.2$, $A=0.8$, $b=0.5$, $h=0.7$, with $s = 0.3, 0.73, 0.78$. 
Right panel (strong Allee effect): $r_1=1.2$, $k_1=5$, $k_0=1$, $\lambda=1.2$, $A=0.2$, $b=0.5$, $h=0.9$, with $s = 0.38, 0.6, 0.8$. 
Here, $g_1^{(2)}, g_2^{(2)}, g_3^{(2)}$ denote the graphs of $g^{(2)}(x)$ corresponding to the three values of $s$ in each case. 
The intersections illustrate that the system admits 0, 1, or 2 interior equilibria, respectively.}
     \label{figure1}
  \end{center}
\end{figure}
 
The following theorem establishes the existence and number of interior equilibria, i.e., equilibrium points in \( \mathbb{R}_+^2 \), for system~\eqref{equ02}.

\begin{theorem}\label{two_equilibria}
    Suppose \(sh< 1\). If one of the following conditions holds:
    \begin{align}
        & k_0 > \tfrac{k_1}{2}, \label{2.4.1}\\
        & k_0 < 0, \ \lambda < Ab, \label{2.4.2}\\
        & k_0 < 0, \ \lambda > Ab, \ \tfrac{r_1}{s}\Bigl(1+\tfrac{k_0}{k_1}\Bigr) < \tfrac{A(\lambda-Ab)(1-sh)^2}{s^2}, \label{2.4.3}
    \end{align}
    then system~\eqref{equ02} admits at most two distinct interior equilibria in \( \mathbb{R}_+^2 \).
\end{theorem}

\begin{proof}
We first consider the case \eqref{2.4.1}; the arguments for \eqref{2.4.2} and \eqref{2.4.3} are analogous and are omitted.

Any interior equilibrium \( (x^*, y^*) \in \mathbb{R}_+^2 \) must satisfy
\begin{equation}\label{equorder}
    \begin{cases}
    y^* = g^{(1)}(x^*) = \dfrac{r_1}{s}x^*\!\left(1-\tfrac{x^*}{k_1}\right)(x^*-k_0) > 0, \\[1.2ex]
    y^* = g^{(2)}(x^*) = \dfrac{\lambda}{A} \left(-1 + \dfrac{s\bigl(1 - \tfrac{A}{\lambda}b\bigr)}{A(sh-1)x^*+s} \right) > 0.
    \end{cases}
\end{equation}
Hence \(x^*\) must lie in
\[
   \Delta := \bigl(\max(k_0,0),\, k_1\bigr) \cap \{x: g^{(2)}(x)>0\}.
\]

For \( g^{(1)}(x) \), we compute
\[
    \frac{d g^{(1)}(x)}{d x} = \frac{r_1}{s}x\!\left(1 - \frac{2x}{k_1} + \frac{k_0}{k_1}\right) 
    + \frac{r_1}{s}\!\left(1 - \frac{x}{k_1}\right)(x - k_0),
\]
and hence
\[
    \frac{d^2 g^{(1)}(x)}{d x^2} = \frac{2r_1}{s}\!\left(1+\frac{k_0}{k_1}\right) - \frac{6r_1}{sk_1}x.
\]
Therefore,
\[
   \frac{d^2 g^{(1)}(x)}{dx^2} < 0 \quad \text{for all } x \in [k_0, k_1],
\]
so \(g^{(1)}(x)\) is concave on \([k_0,k_1]\).

For \( g^{(2)}(x) \), we consider two cases:

- If \( \lambda - Ab \geq 0 \), then for all \(x \in \Delta\), \( A(sh - 1)x + s > 0 \) and 
\begin{equation}\label{g2derivatives}
\begin{aligned}
    \frac{d g^{(2)}(x)}{dx} &= \frac{s(\lambda - Ab)(1 - sh)}{\bigl[A(sh - 1)x + s\bigr]^2} \geq 0, \\
    \frac{d^2 g^{(2)}(x)}{dx^2} &= \frac{2As(\lambda - Ab)(1 - sh)^2}{\bigl[A(sh - 1)x + s\bigr]^3} \geq 0.
\end{aligned}
\end{equation}

- If \( \lambda - Ab < 0 \), then by \eqref{g2derivatives}, for all \(x \in \Delta\),
\[
   \frac{d g^{(2)}(x)}{dx} < 0, \qquad 
   \frac{d^2 g^{(2)}(x)}{dx^2} > 0.
\]
Thus, in either case, \( g^{(2)}(x) \) is convex on \( \Delta \).

Now define \( g(x) := g^{(2)}(x) - g^{(1)}(x) \). Since \( g^{(1)} \) is concave and \( g^{(2)} \) is convex on \( \Delta \), the function \( g(x) \) can have at most one critical point in \( \Delta \). Consequently, \( g(x) \) admits at most two zeros in \( \Delta \). Hence, system~\eqref{equ02} possesses at most two interior equilibria.

    This completes the proof.
\end{proof}

\section{Stability of Equilibria}\label{sec3}

In this section, we investigate the stability properties of the equilibria for system \eqref{equ02}. By analyzing the Jacobian matrix evaluated at each boundary and interior equilibrium, we determine their stability characteristics. Additionally, we derive sufficient conditions for the global stability of the system and explore scenarios leading to population decline.

The Jacobian matrix of system \eqref{equ02} at an equilibrium point \(E(x, y)\) is given by:
\begin{equation} \label{eqj(e)}
J(E) =
\begin{pmatrix}
\frac{\partial f_1}{\partial x} & \frac{\partial f_1}{\partial y} \\
\frac{\partial f_2}{\partial x} & \frac{\partial f_2}{\partial y}
\end{pmatrix},
\end{equation}
where
\begin{equation} \label{eq_f1f1}
\begin{aligned}
    f_1(x,y) &= x f^{(1)}(x,y) = r_1x\left(1 - \frac{x}{k_1}\right)(x - k_0) - \frac{(\lambda+Ay)xy}{b+y+hx(\lambda+Ay)}, \\
    f_2(x,y) &= y f^{(2)}(x,y) = \frac{(\lambda+Ay)xy}{b+y+hx(\lambda+Ay)} - sy.
\end{aligned}
\end{equation}

The first and second partial  derivatives of \( f_1 \) and \( f_2 \) are given by:
\begin{align}
    \frac{\partial f_1}{\partial x} &= r_1(1 - \frac{x}{k_1})(x - k_0) + r_1x\left(1 + \frac{k_0}{k_1} - \frac{2x}{k_1} \right) - \frac{y(b+y)(\lambda+Ay)}{[b+y+hx(\lambda+Ay)]^2}, \label{f1x}  \\
    \frac{\partial f_2}{\partial x} &= \frac{y(b+y)(\lambda+Ay)}{[b+y+hx(\lambda+Ay)]^2}, \label{f2_derivative} \\
    \frac{\partial f_1}{\partial y} &= -\frac{\lambda x[b+hx(\lambda+Ay)] + 2Abxy + Axy^2 + Ahx^2y(\lambda+Ay)}{[b+y+hx(\lambda+Ay)]^2},\label{f1_derivative}  \\
    \frac{\partial f_2}{\partial y} &= \frac{\lambda x[b+hx(\lambda+Ay)] + 2Abxy + Axy^2 + Ahx^2y(\lambda+Ay)}{[b+y+hx(\lambda+Ay)]^2} - s,\label{f2y} \\
    \frac{\partial^2 f_1}{\partial x^2} &= 
    r_1 \left( 2 + 2\frac{k_0}{k_1} - \frac{6}{k_1}x \right) 
    + \frac{2y(b+y)(\lambda+Ay)^2h}{[b+y+hx(\lambda+Ay)]^3}, \label{f1xx}\\
    \frac{\partial^2 f_2}{\partial x^2} &= -
    \frac{2y(b+y)(\lambda+Ay)^2h}{[b+y+hx(\lambda+Ay)]^3},   \label{f2xx}\\
    \frac{\partial^2 f_1}{\partial y^2} & = 
    2\frac{x(Ab-\lambda)(1+hAx)y}{[b+y+h x(\lambda+A y)]^3}-\frac{2x(Ab-\lambda)}{(b+y+hx(\lambda + Ay))^2}, \label{f1yy}\\
    \frac{\partial^2 f_2}{\partial y^2} &= 
    -\frac{\partial^2 f_1}{\partial y^2} = 
    2\frac{\partial f^{(2)}}{\partial y} + y\frac{\partial^2 f^{(2)}}{\partial y^2} \\[4pt]
    &= \frac{2x(Ab-\lambda)}{(b+y+hx(\lambda + Ay))^2} 
    - \frac{2x(Ab-\lambda)(1+hAx)y}{[b+y+h x(\lambda+A y)]^3}, \\
    \frac{\partial^2 f_1}{\partial x \partial y} &=-\frac{1}{[b+y+h x(\lambda+A y)]^3}  
    \Big[(\lambda b+ 2Ab y+2\lambda y+3A y^2)[b+y+h x(\lambda+A y)]   \notag\\
    &\quad - 2y(\lambda b + Ab y+\lambda y+A y^2)(1+hA x)\Big], \label{f1xy} \\
    \frac{\partial^2 f_2}{\partial x \partial y} 
    & =
    -\frac{\partial^2 f_1}{\partial x \partial y} \notag \\
    &= \frac{1}{[b+y+h x(\lambda+A y)]^3} \times \Big[(\lambda b+ 2Ab y+2\lambda y+3A y^2)(b+y+h x(\lambda+A y)) \notag \\
    & \quad    - 2y(\lambda b + Ab y+\lambda y+A y^2)(1+hA x)\Big]. \label{f2xy}
\end{align}

It is observed that system \eqref{equ02} has three boundary equilibria under the strong Allee effect (\( k_0 > 0 \)):
\begin{equation} \label{E012}
E_0 (0,0), \quad E_1 (k_1, 0), \quad E_2 (k_0, 0).
\end{equation}
In contrast, under the weak Allee effect (\( k_0 < 0 \)), only two boundary equilibria exist in the non-negative quadrant:
\[
E_0 (0,0) \quad \text{and} \quad E_1 (k_1, 0).
\]

In the following theorems, we analyze the stability of these boundary equilibria.
\begin{theorem}\label{thm03}
The stability of the boundary equilibria of system~\eqref{equ02} is characterized as follows:
\begin{itemize}
    \item The equilibrium point \( E_0 \) satisfies:
    \begin{itemize}
        \item If \( k_0 > 0 \) (strong Allee effect), then \( E_0 \) is a locally asymptotically stable hyperbolic equilibrium.
        \item If \( k_0 < 0 \) (weak Allee effect), then \( E_0 \) is a hyperbolic saddle.
    \end{itemize}

    \item The equilibrium point \( E_1 \) satisfies:
    \begin{itemize}
        \item If \( \displaystyle \frac{\lambda k_1}{b + h k_1 \lambda} - s < 0 \), then \( E_1 \) is a stable hyperbolic equilibrium.
        \item If \( \displaystyle \frac{\lambda k_1}{b + h k_1 \lambda} - s > 0 \), then \( E_1 \) is a hyperbolic saddle.
        \item If \( \displaystyle \frac{\lambda k_1}{b + h k_1 \lambda} - s = 0 \), then \( E_1 \) is a degenerate equilibrium.
    \end{itemize}

    \item If \( k_0 > 0 \), then the equilibrium \( E_2  \) lies in the biologically relevant domain and:
    \begin{itemize}
        \item If \( \displaystyle \frac{\lambda k_0}{b + h k_0 \lambda} - s < 0 \), then \( E_2 \) is a hyperbolic saddle.
        \item If \( \displaystyle \frac{\lambda k_0}{b + h k_0 \lambda} - s > 0 \), then \( E_2 \) is an unstable hyperbolic equilibrium.
        \item If \( \displaystyle \frac{\lambda k_0}{b + h k_0 \lambda} - s = 0 \), then \( E_2 \) is a degenerate equilibrium.
    \end{itemize}
\end{itemize}
\end{theorem}
\begin{proof}
We analyze the stability of each boundary equilibrium by evaluating the Jacobian matrix of system~\eqref{equ02} at the corresponding points.

\textbf{1. Stability of \( E_0 = (0,0) \).}
The Jacobian matrix at \( E_0 \) is:
\[
J(E_0) = 
\begin{pmatrix}
\displaystyle \frac{\partial f_1}{\partial x} & \displaystyle \frac{\partial f_1}{\partial y} \\
\displaystyle \frac{\partial f_2}{\partial x} & \displaystyle \frac{\partial f_2}{\partial y}
\end{pmatrix}_{(0,0)} 
= 
\begin{pmatrix}
- r_1 k_0 & 0 \\
0 & -s
\end{pmatrix}.
\]
The eigenvalues of \( J(E_0) \) are \( -r_1 k_0 \) and \( -s \). Note that  \( s > 0 \) and  the sign of \( -r_1 k_0 \) depends on \( k_0 \):
\begin{itemize}
    \item If \( k_0 > 0 \) (strong Allee effect), both eigenvalues are negative. Hence, \( E_0 \) is a stable hyperbolic equilibrium.
    \item If \( k_0 < 0 \) (weak Allee effect), one eigenvalue is positive and the other is negative. Hence, \( E_0 \) is a hyperbolic saddle.
\end{itemize}

\textbf{2. Stability of \( E_1 = (k_1, 0) \).}
Compute the Jacobian matrix:
    \[
    \frac{\partial f_1}{\partial x} = r_1 k_1 \left(\frac{k_0}{k_1} - 1\right), \quad \frac{\partial f_1}{\partial y} = -\frac{\lambda k_1}{b + h k_1 \lambda}, \quad \frac{\partial f_2}{\partial x} = 0, \quad \frac{\partial f_2}{\partial y} = \frac{\lambda k_1}{b + h k_1 \lambda} - s.
    \]
Therefore, the Jacobian at \( E_1 \) is:
\[
J(E_1) = 
\begin{pmatrix}
r_1 k_1 \left( \frac{k_0}{k_1} - 1 \right) & - \frac{\lambda k_1}{b + h k_1 \lambda} \\
0 & \frac{\lambda k_1}{b + h k_1 \lambda} - s
\end{pmatrix}.
\]
This is a triangular matrix, so its eigenvalues are:
\[
\zeta_1 = r_1 k_1 \left( \frac{k_0}{k_1} - 1 \right), \quad 
\zeta_2 = \frac{\lambda k_1}{b + h k_1 \lambda} - s.
\]
Note that \( k_0 < k_1 \), so \( \zeta_1 < 0 \) always. The sign of \( \zeta_2 \) determines the nature of \( E_1 \):
\begin{itemize}
    \item If \( \zeta_2 < 0 \), then both eigenvalues are negative, and \( E_1 \) is a stable hyperbolic equilibrium.
    \item If \( \zeta_2 > 0 \), then one eigenvalue is negative and one is positive, so \( E_1 \) is a hyperbolic saddle.
    \item If \( \zeta_2 = 0 \), then \( E_1 \) is a degenerate equilibrium.

\textbf{3. Stability of \( E_2 = (k_0, 0) \).}

We now consider the case when \( k_0 > 0 \) (i.e., \( E_2 \in \mathbb{R}_+^2 \)). At \( E_2 = (k_0, 0) \), we compute:
    \[
    \frac{\partial f_1}{\partial x} = r_1 k_0 \left(1 - \frac{k_0}{k_1} \right), \quad \frac{\partial f_1}{\partial y} = -\frac{\lambda k_0}{b + h k_0 \lambda}, \quad \frac{\partial f_2}{\partial x} = 0, \quad \frac{\partial f_2}{\partial y} = \frac{\lambda k_0}{b + h k_0 \lambda} - s.
    \]
Hence, the Jacobian at \( E_2 \) is:
\[
J(E_2) = 
\begin{pmatrix}
r_1 k_0 \left(1 - \frac{k_0}{k_1} \right) & - \frac{\lambda k_0}{b + h k_0 \lambda} \\
0 & \frac{\lambda k_0}{b + h k_0 \lambda} - s
\end{pmatrix}.
\]
The eigenvalues are:
\[
\zeta_3 = r_1 k_0 \left(1 - \frac{k_0}{k_1} \right) > 0, \quad 
\zeta_4 = \frac{\lambda k_0}{b + h k_0 \lambda} - s.
\]
Thus, the stability of \( E_2 \) depends on \( \zeta_4 \):
\begin{itemize}
    \item If \( \zeta_4 < 0 \), then one eigenvalue is positive and one is negative, so \( E_2 \) is a hyperbolic saddle.
    \item If \( \zeta_4 > 0 \), both eigenvalues are positive, so \( E_2 \) is an unstable hyperbolic equilibrium.
    \item If \( \zeta_4 = 0 \), then \( E_2 \) is a degenerate equilibrium.
\end{itemize}

 The classification of the boundary equilibria under both strong and weak Allee effect follows directly from the eigenvalue analysis of the Jacobian matrices at \( E_0 \), \( E_1 \), and (when \( k_0 > 0 \)) \( E_2 \), as described above.     We complete the proof.
\end{itemize}
\end{proof}

Let us now examine the stability of interior equilibria, denoted by \( E^*(x^*, y^*) \), of system \eqref{equ02}. To do so, we analyze the Jacobian matrix at \( E^* \), using \eqref{eq_f1f1}:
\begin{align}\label{mtx01}
     J(E^*)=
\begin{pmatrix}
\frac{\partial f_1}{\partial x} & \frac{\partial f_1}{\partial y}\\
\frac{\partial f_2}{\partial x} & \frac{\partial f_2}{\partial y}
\end{pmatrix}
=
\begin{pmatrix}
f^{(1)} + x\frac{\partial f^{(1)}}{\partial x} &x \frac{\partial f^{(1)}}{\partial y}\\
y\frac{\partial f^{(2)}}{\partial x} & f^{(2)} + y\frac{\partial f^{(2)}}{\partial y}
\end{pmatrix}.
\end{align}

Let {\( y_1(x)>0 \)} and \( y_2(x) \) denote the solutions of \( f^{(1)}(x,y) = 0 \) and \( f^{(2)}(x,y) = 0 \), respectively.  
Applying the implicit function theorem, we obtain
{\begin{align}\label{E09}
    \frac{d y_i}{d x} & = - \frac{\frac{\partial f^{(i)}}{\partial x}}{\frac{\partial f^{(i)}}{\partial y}}, \qquad i=1, 2.
\end{align}
}

Since \( E^* \) is an interior equilibrium, it satisfies \( f^{(1)}(E^*) = f^{(2)}(E^*) = 0 \). Substituting this property into the Jacobian matrix, we can rewrite \( J(E^*) \) as:
\begin{align}\label{mtx02}
     J(E^*)=
\begin{pmatrix}
-x\frac{\partial f^{(1)}}{\partial y}\frac{d y_1}{d x} &x \frac{\partial f^{(1)}}{\partial y}\\
-y\frac{\partial f^{(2)}}{\partial y}\frac{d y_2}{d x} & y\frac{\partial f^{(2)}}{\partial y}
\end{pmatrix}.
\end{align}

With straightforward calculations, we derive the following lemma.
\begin{lemma}\label{lem02}
    The determinant and trace of the Jacobian matrix evaluated at $E^*$ can be expressed as
    $$\det(J(E^*)) = xy\frac{\partial f^{(1)}}{\partial y}\frac{\partial f^{(2)}}{\partial y}\big(\frac{d y_2}{d x} - \frac{d y_1}{d x}  \big)\mid_{(x^*,y^*)},$$
    $$Tr(J(E^*))  = y\frac{\partial f^{(2)}}{\partial y} - x\frac{\partial f^{(1)}}{\partial y}\frac{d y_1}{d x} \mid_{(x^*,y^*)}= x\frac{\partial f^{(1)}}{\partial x} + y\frac{\partial f^{(2)}}{\partial y}\mid_{(x^*,y^*)}.$$
\end{lemma}

Recall that Theorem \ref{two_equilibria} establishes that under some conditions,  system \eqref{equ02} has at most two distinct interior equilibria in \( \mathbb{R}_+^2 \). The stability properties of these equilibria are summarized in the following theorem.

\begin{theorem}\label{thm01}
Consider system~\eqref{equ02} with initial condition $(x(0),y(0)) \in \mathbb{R}_+^2$.  
Let $E^*=(x^*,y^*)$ be an interior equilibrium, and define
\[
y_1(x) \;=\; \frac{f(x)(1+hAx)+\lambda + \sqrt{\,\bigl(f(x)(1+hAx)+\lambda\bigr)^2+4A(b+h\lambda x)f(x)}}{2A},
\]
$$y_2(x) \;=\; \frac{\lambda}{A}\left[-1+\frac{s\bigl(1-\tfrac{A}{\lambda}b\bigr)}{A(sh-1)x+s}\right],$$
where
\[
f(x) = r_1\!\left(1-\frac{x}{k_1}\right)(x-k_0).
\]

Then the local stability of $E^*$ is determined as follows:

\begin{itemize}
    \item[(a)] If $sh<1$ and $\lambda > A b$:  
    \begin{itemize}
        \item If $\tfrac{dy_2}{dx}-\tfrac{dy_1}{dx}<0$ at $E^*$, then $E^*$ is a saddle point.  
        \item If $\tfrac{dy_2}{dx}-\tfrac{dy_1}{dx}>0$ at $E^*$, then $E^*$ is locally stable when $\operatorname{Tr}(J(E^*))<0$, and unstable when $\operatorname{Tr}(J(E^*))>0$.
    \end{itemize}

    \item[(b)] If $sh<1$ and $\lambda < A b$:  
    \begin{itemize}
        \item If $\tfrac{dy_2}{dx}-\tfrac{dy_1}{dx}>0$ at $E^*$, then $E^*$ is a saddle point.  
        \item If $\tfrac{dy_2}{dx}-\tfrac{dy_1}{dx}<0$ at $E^*$, then $E^*$ is locally stable when $\operatorname{Tr}(J(E^*))<0$, and unstable when $\operatorname{Tr}(J(E^*))>0$.
    \end{itemize}
\end{itemize}
\end{theorem}

\begin{proof}
From \eqref{eq_f1f1}, the partial derivatives of the vector field at $E^*$ are
\[
\frac{\partial f^{(1)}}{\partial y} \;=\; \frac{1}{x}\frac{\partial f_1}{\partial y}, 
\qquad 
\frac{\partial f^{(2)}}{\partial y} \;=\; \frac{A b - \lambda}{\bigl[b + y + h x (\lambda + A y)\bigr]^2}.
\]
By \eqref{f1_derivative}, we know that $\tfrac{\partial f^{(1)}}{\partial y}<0$, while the sign of $\tfrac{\partial f^{(2)}}{\partial y}$ depends on whether $\lambda$ is greater or smaller than $A b$.

\medskip
\noindent\textbf{Case (a)}: $sh<1$ and $\lambda > A b$.

Here $\tfrac{\partial f^{(2)}}{\partial y}<0$.  
Applying Lemma~\ref{lem02}, the determinant of the Jacobian at $E^*$ is
\[
\det \bigl(J(E^*)\bigr) 
= x^* y^* \,\frac{\partial f^{(1)}}{\partial y}(E^*)\,\frac{\partial f^{(2)}}{\partial y}(E^*) 
\left(\tfrac{dy_2}{dx}-\tfrac{dy_1}{dx}\right)\Big|_{E^*}.
\]

\begin{itemize}
    \item If $\tfrac{dy_2}{dx}-\tfrac{dy_1}{dx}<0$ at $E^*$, then $\det(J(E^*))<0$, so $E^*$ is a saddle point.
    \item If $\tfrac{dy_2}{dx}-\tfrac{dy_1}{dx}>0$ at $E^*$, then $\det(J(E^*))>0$. In this case, the type of $E^*$ is determined by the trace: it is locally stable when $\operatorname{Tr}(J(E^*))<0$, and unstable when $\operatorname{Tr}(J(E^*))>0$.
\end{itemize}

\noindent\textbf{Case (b)}: $sh<1$ and $\lambda < A b$.

In this case $\tfrac{\partial f^{(2)}}{\partial y}>0$.  
By the same argument as in Case~(a), the sign of $\det(J(E^*))$ is reversed, and the conclusions follow analogously:
\begin{itemize}
    \item If $\tfrac{dy_2}{dx}-\tfrac{dy_1}{dx}>0$ at $E^*$, then $\det(J(E^*))<0$, hence $E^*$ is a saddle point.
    \item If $\tfrac{dy_2}{dx}-\tfrac{dy_1}{dx}<0$ at $E^*$, then $\det(J(E^*))>0$, and $E^*$ is locally stable when $\operatorname{Tr}(J(E^*))<0$, but unstable when $\operatorname{Tr}(J(E^*))>0$.
\end{itemize}

This completes the proof.
\end{proof}

In the remainder of this section, we analyze the global dynamics of system \eqref{equ02}. We not only examine global stability but also provide sufficient conditions for the decline of predators. 

\begin{theorem}\label{globaldynamics}
Suppose the parameters satisfy
\begin{equation*}
\begin{cases}
 \lambda > A b, \\[0.3ex]
 A k_1 (s h - 1) + s > 0, \\[0.3ex]
 (1 - s h)\,\lambda k_1 > s b, \\[0.3ex]
 \dfrac{r_1}{s}\!\left(1+\dfrac{k_0}{k_1}\right) < \dfrac{A(\lambda - A b)(1 - s h)^2}{s^2}.
\end{cases}
\end{equation*}
Furthermore, assume that for every $(x,y)$ in
\[
\mathbf{D}^{0}:=(0,\,k_1)\times \Bigl(0,\,\max\bigl\{\,k_1,\,\max_{z\in[0,k_1]} g^{(1)}(z)\,\bigr\}\Bigr),
\]
the Jacobian matrix $J$ defined in~\eqref{eqj(e)} satisfies
\[
\operatorname{Tr}\bigl(J(x,y)\bigr)
= x\,\frac{\partial f^{(1)}}{\partial x}(x,y)
+ y\,\frac{\partial f^{(2)}}{\partial y}(x,y) < 0.
\]
Then, under the weak Allee effect, system~\eqref{equ02} possesses a unique interior equilibrium, which is globally asymptotically stable.
\end{theorem}

\begin{proof}
Since \( (1 - sh)\lambda k_1 > sb \), it follows that \( sh < 1 \).  
Moreover, by the assumptions \(\lambda > Ab\) and 
\(\tfrac{r_1}{s}\!\left(1+\tfrac{k_0}{k_1}\right) < \tfrac{A(\lambda-Ab)(1-sh)^2}{s^2}\),  
Theorem~\ref{two_equilibria} guarantees that there are at most two intersections between 
\(y_1(x)\) and \(y_2(x)\). From~\eqref{equ22}, we compute
\[
    g^{(2)}(k_1) = \frac{(1 - s h)\lambda k_1 - sb}{s + A(sh - 1)k_1} > 0,
    \qquad g^{(2)}(0) = -b < 0,
\]
while from~\eqref{equ21} we have \( g^{(1)}(0) = g^{(1)}(k_1) = 0 \).  
Hence, by the intermediate value theorem, the equation \( g^{(1)}(x) = g^{(2)}(x) \) admits a solution in \((0, k_1)\). Consequently, there exists a unique interior equilibrium 
\( E^* = (x^*,y^*) \) with \( x^* \in (0,k_1) \) and \( y^* = g^{(1)}(x^*) > 0 \), so that
\[
   y^* \in \Bigl(0,\,\max_{z\in[0,k_1]} g^{(1)}(z)\Bigr),
   \quad\text{and hence } E^* \in \mathbf{D}^0.
\]

\noindent\emph{Local stability.}  
From the preceding analysis,
\[
   y_2(0) = -b < 0, \quad y_2(k_1) > 0, \quad y_1(0) = y_1(k_1) = 0,
\]
which implies that at \(E^*\) we have
\[
   \frac{d y_2}{dx} > \frac{d y_1}{dx}.
\]
Together with the assumption
\[
   \operatorname{Tr}\bigl(J(x,y)\bigr)
   = x\,\frac{\partial f^{(1)}}{\partial x}(x,y)
   + y\,\frac{\partial f^{(2)}}{\partial y}(x,y) < 0 \quad \text{at } E^*,
\]
Theorem~\ref{thm01} ensures that \(E^*\) is locally asymptotically stable.

\medskip
\noindent\emph{Nonexistence of periodic orbits.}  
Define the Dulac function \( D(x,y) = 1/(xy) \). Then
\begin{equation}\label{equ09}
\begin{aligned}
\frac{\partial (D f_1)}{\partial x} + \frac{\partial (D f_2)}{\partial y} 
&= D \left( \frac{\partial f_1}{\partial x} + \frac{\partial f_2}{\partial y} \right) 
   + f_1 \frac{\partial D}{\partial x} + f_2 \frac{\partial D}{\partial y} \\[4pt]
&= D \left( f^{(1)} + x \frac{\partial f^{(1)}}{\partial x} 
            + f^{(2)} + y \frac{\partial f^{(2)}}{\partial y} \right) 
   - \frac{Df_1}{x} - \frac{Df_2}{y} \\[4pt]
&= D \left( x \frac{\partial f^{(1)}}{\partial x} 
           + y \frac{\partial f^{(2)}}{\partial y} \right) \\[4pt]
&= D \,\operatorname{Tr}\bigl(J(x,y)\bigr) < 0, 
   \quad (x,y) \in \mathbf{D}^0.
\end{aligned}
\end{equation}
By Dulac’s criterion~\cite{burton2005volterra}, system~\eqref{equ02} admits no periodic orbits in \(\mathbf{D}^0\).  
Moreover, Theorem~\ref{lem08} implies that for any initial condition 
\((x(0),y(0)) \in \mathbb{R}_+^2\), the trajectory remains bounded:
\[
   \limsup_{t\to\infty} x(t) \leq k_1, \qquad 
   \limsup_{t\to\infty} y(t) \leq k_1.
\]
Thus, the system has no periodic solutions in \(\mathbb{R}_+^2\).

\medskip
\noindent\emph{Dynamics on the boundary.}  
Since \((1-sh)\lambda k_1 - sb > 0\), it follows that
\(\tfrac{\lambda k_1}{b + h k_1 \lambda} - s > 0\). Hence, the boundary equilibrium 
\(E_1 = (k_1,0)\) is a saddle by Theorem~\ref{thm03}, stable along the \(x\)-axis.  
Similarly, \(E_0 = (0,0)\) is a saddle, stable along the \(y\)-axis.

\medskip
\noindent\emph{Global stability.}  
Since the set \(\mathbf{B}^\delta\) (\(\delta > 0\)) is positively invariant by Theorem~\ref{lem08}, and the system has no periodic orbits in \(\mathbb{R}_+^2\), the Poincaré–Bendixson theorem implies that every trajectory in \(\mathbb{R}_+^2\) converges to an equilibrium.  
As the only interior equilibrium is \(E^*\), it follows that \(E^*\) is globally asymptotically stable.

This completes the proof.
\end{proof}

In Theorem~\ref{globaldynamics}, we assume that the Jacobian trace satisfies 
\(\operatorname{Tr}(J(x,y)) < 0\) for all \((x,y) \in \mathbf{D}^0\).  
The following lemma provides a sufficient parameter condition under which this assumption holds.

\begin{lemma}\label{traceJacobian}
Suppose system~\eqref{equ02} exhibits a weak Allee effect and that \(k_0 = -k_1\).  
If the parameters satisfy
\[
   (Ab - \lambda) + \Bigl(\lambda + A \max\{\,k_1, \max_{z \in [0,k_1]} g^{(1)}(z)\,\}\Bigr)^{\!2} h < 0,
\]
then the Jacobian trace
\[
   \operatorname{Tr}\bigl(J(x,y)\bigr) 
   = x\,\frac{\partial f^{(1)}}{\partial x}(x,y) 
   + y\,\frac{\partial f^{(2)}}{\partial y}(x,y)
\]
is strictly negative for all \((x,y) \in \mathbf{D}^0\).
\end{lemma}

\begin{proof}
We first compute the relevant partial derivatives:
\begin{align*}
    \frac{\partial f^{(1)}}{\partial x} 
        &= r_1 \left( 1 + \frac{k_0}{k_1} - \frac{2x}{k_1} \right) 
           + \frac{(\lambda + Ay)^2 h y}{\bigl(b + y + h x (\lambda + Ay)\bigr)^2},\\[6pt]
    \frac{\partial f^{(2)}}{\partial y} 
        &= \frac{x (Ab - \lambda)}{\bigl(b + y + h x (\lambda + Ay)\bigr)^2}.
\end{align*}

Hence, the Jacobian trace is
\begin{align*}
    \operatorname{Tr}(J(x, y)) &= x \frac{\partial f^{(1)}}{\partial x} + y \frac{\partial f^{(2)}}{\partial y} \\
    &= r_1 x \left( 1 + \frac{k_0}{k_1} - \frac{2x}{k_1} \right) + \frac{(\lambda + Ay)^2 h x y}{\left(b + y + h x (\lambda + Ay)\right)^2} + \frac{(Ab - \lambda) x y}{\left(b + y + h x (\lambda + Ay)\right)^2} \\
    &= r_1 x \left( 1 + \frac{k_0}{k_1} - \frac{2x}{k_1} \right) + \frac{x y \left[ (\lambda + Ay)^2 h + (Ab - \lambda) \right]}{\left(b + y + h x (\lambda + Ay)\right)^2}.
\end{align*}

Substituting \( k_0 = -k_1 \) gives
\[
    r_1 x \left( 1 - \frac{k_1}{k_1} - \frac{2x}{k_1} \right) 
    = -\frac{2 r_1 x^2}{k_1},
\]
so that
\[
    \operatorname{Tr}(J(x, y)) 
    = -\frac{2 r_1 x^2}{k_1} 
      + \frac{x y \bigl[(\lambda + Ay)^2 h + (Ab - \lambda)\bigr]}{\bigl(b + y + h x (\lambda + Ay)\bigr)^2}.
\]
Since the first term is nonpositive, it follows that
\[
    \operatorname{Tr}(J(x, y)) 
    \leq \frac{x y \bigl[(\lambda + Ay)^2 h + (Ab - \lambda)\bigr]}{\bigl(b + y + h x (\lambda + Ay)\bigr)^2}.
\]

By the parameter condition
\[
   (Ab - \lambda) + \Bigl(\lambda + A \max\{k_1,\max_{z \in [0,k_1]} g^{(1)}(z)\}\Bigr)^{\!2} h < 0,
\]
we obtain
\[
   (\lambda + Ay)^2 h + (Ab - \lambda) < 0, 
   \qquad \text{for all } (x,y) \in \mathbf{D}^0.
\]

Since \(x > 0, y > 0\), and the denominator is positive in \(\mathbf{D}^0\), the inequality above implies
\[
    \operatorname{Tr}(J(x, y)) < 0, 
    \qquad \text{for all } (x, y) \in \mathbf{D}^0.
\]
This proves the lemma.
\end{proof}

Combining Theorem~\ref{globaldynamics} with Lemma~\ref{traceJacobian}, we obtain the following corollary.

\begin{corollary}\label{cor1}
Assume that \( k_0 = -k_1 \). If the parameters satisfy
\[
\lambda - Ab > 0, 
\qquad Ak_1 - s < 0, 
\qquad \frac{r_1}{s}\left(1+\frac{k_0}{k_1}\right) < \frac{A(\lambda-Ab)(1-sh)^2}{s^2},
\]
and
\[
h < \min \left\{ \frac{\lambda - Ab}{\left(\lambda + A \max\{\,k_1,\; \max_{z \in [0,k_1]} g^{(1)}(z)\} \right)^2}, \,
\frac{\lambda k_1 - s b}{s \lambda k_1} \right\},
\]
then system~\eqref{equ02} admits a unique globally asymptotically stable interior equilibrium.
\end{corollary}

In Corollary~\ref{cor1}, the assumption \( k_0 = -k_1 \) corresponds to the weak Allee effect. 
In contrast, under a strong Allee effect, system~\eqref{equ02} does not possess a globally stable interior equilibrium: the boundary equilibrium \( E_0 \) is stable by Theorem~\ref{thm03}. 
Nevertheless, in this regime, the system may admit a unique interior equilibrium which is unstable. 
The next theorem establishes conditions under which this occurs and shows that the predator population ultimately goes extinct.

\begin{theorem} \label{thm3.6}
Suppose system~\eqref{equ02} is subject to a strong Allee effect with \(1 - sh > 0\) and \(k_0 > \tfrac{k_1}{2}\).  
If either of the following conditions holds:
\begin{itemize}
    \item[(i)] \( \lambda - Ab < 0 \), \( g^{(2)}(k_1) < 0 \), \( g^{(2)}(k_0) > 0 \), and \( \operatorname{Tr}(J(x,y)) > 0 \) for all \( (x,y) \in \mathbf{D}^0 \);
    
    \item[(ii)] \( \lambda - Ab > 0 \), \( g^{(2)}(k_1) > 0 \), \( g^{(2)}(k_0) < 0 \), and \( \operatorname{Tr}(J(x,y)) > 0 \) for all \( (x,y) \in \mathbf{D}^0 \),
\end{itemize}
then system~\eqref{equ02} admits a unique interior equilibrium, which is unstable.  
Moreover, for any initial condition \( (x_0, y_0) \in \mathbb{R}_+^2 \), the solution satisfies
\[
   \lim_{t \to \infty} y(t) = 0,
\]
that is, the predator population goes extinct asymptotically.
\end{theorem}

\begin{proof}
We prove the result for case (i); the argument for case (ii) is analogous and is omitted.  

From \eqref{equ21}, we have
\[
g^{(1)}(k_0) = g^{(1)}(k_1) = 0.
\]
Since \( g^{(2)}(k_1) < 0 \) and \( g^{(2)}(k_0) > 0 \), the intermediate value theorem implies that the equation
\[
g^{(1)}(x) = g^{(2)}(x)
\]
has a unique solution in \( (k_0, k_1) \). Thus system~\eqref{equ02} admits at least one interior equilibrium.  

On the other hand, because \(1 - sh > 0\) and \(k_0 > \tfrac{k_1}{2}\), Theorem~\ref{two_equilibria} guarantees that there can be at most two interior equilibria. Combining these facts, we conclude that system~\eqref{equ02} possesses a unique interior equilibrium.  

Next, we analyze its stability. Since \(\operatorname{Tr}(J(x,y)) > 0\) for all \((x,y) \in \mathbf{D}^0\), Theorem~\ref{thm01} implies that this interior equilibrium is unstable.

It remains to establish the asymptotic behavior of trajectories. Consider the Dulac function \(D(x,y) = 1/(xy)\), as in the proof of Theorem~\ref{globaldynamics}. Then
\[
\frac{\partial (D f_1)}{\partial x} + \frac{\partial (D f_2)}{\partial y} 
= D(x,y) \, \operatorname{Tr}(J(x,y)) > 0,
\]
for all \((x,y) \in \mathbf{D}^0\). By Dulac's criterion, system~\eqref{equ02} admits no periodic orbits in \(\mathbf{D}^0\).  

Since the system has a unique unstable interior equilibrium, admits no periodic solutions, and evolves within the positively invariant compact set \(\mathbf{B}^\delta\) (Theorem~\ref{lem08}), the Poincaré–Bendixson theorem implies that every trajectory converges to one of the boundary equilibria:  
\[
E_0 = (0,0), \quad E_1 = (k_1,0), \quad \text{or} \quad E_2 = (k_0,0).
\]
In each case, the predator population vanishes, i.e.,
\[
\lim_{t \to \infty} y(t) = 0.
\]
This completes the proof.
\end{proof}

\section{Local and Global Bifurcation Analysis}\label{sec4}

This section investigates the bifurcation structure of system~\eqref{equ02}. We begin by analyzing \textit{local bifurcations}, which occur when small changes in a system parameter cause qualitative changes in the nature or stability of equilibria. Specifically, we consider transcritical, saddle-node, and Hopf bifurcations. In the latter part of the section, we explore \textit{global bifurcations}, including heteroclinic bifurcations and transcritical bifurcations occurring on invariant cycles, which involve global changes in the topology of the phase portrait.

\subsection{Transcritical and Saddle-Node Bifurcations}\label{Transcritical}

A \emph{transcritical bifurcation} is a local bifurcation where two equilibrium branches intersect and exchange their stability as a parameter passes through a critical value. In the generic case, the equilibria coincide at the bifurcation point. In degenerate cases, both equilibria may persist independently of the parameter, but their stability characteristics switch at the bifurcation value.

In contrast, a \emph{saddle-node bifurcation} occurs when two equilibrium branches coalesce into a single non-hyperbolic equilibrium at a critical parameter value. For values of the parameter on one side of the bifurcation, two distinct equilibria exist (typically one saddle and one node); on the other side, no nearby equilibria exist.

The following lemma provides conditions for identifying these bifurcation types.

\begin{lemma}[Transcritical and Saddle-Node Bifurcations \cite{kuznetsov1998elements,perko2013differential}]\label{lem03}
Consider the parameter-dependent system
\[
\dot{X} = f(X, \kappa),
\]
where \( f \in C^2(\mathbb{R}^n \times \mathbb{R}, \mathbb{R}^n) \), \( X \in \mathbb{R}^n \), and \( \kappa \in \mathbb{R} \) is a bifurcation parameter. Suppose the following:
\begin{itemize}
    \item[(a)] \( f(X_0, \mu_0) = 0 \) for some \( (X_0, \mu_0) \in \mathbb{R}^n \times \mathbb{R} \),
    \item[(b)] The Jacobian matrix \( J = D_X f(X_0, \mu_0) \) has a simple eigenvalue \( \gamma = 0 \), and all other eigenvalues have nonzero real parts,
    \item[(c)] Let \( \mathbf{V} \) and \( \mathbf{W} \) be the corresponding right and left eigenvectors (i.e., $J\mathbf{V}=\mathbf{W}^T J=0$), normalized such that \( \mathbf{W}^T \mathbf{V} = 1 \).
\end{itemize}
Then,
\begin{itemize}
    \item[(i)] A transcritical bifurcation occurs at \( (X_0, \mu_0) \) if the following conditions hold:
    \begin{align}
        & \mathbf{W}^T f_{\mu}(X_0, \mu_0) = 0, \label{trans_cond1}\\
        & \mathbf{W}^T [ D_X f_{\mu}(X_0, \mu_0) \, \mathbf{V} ] \neq 0, \label{trans_cond2}\\
        & \mathbf{W}^T [ D_X^2 f(X_0, \mu_0)(\mathbf{V}, \mathbf{V}) ] \neq 0. \label{trans_cond3}
    \end{align}

    \item[(ii)] A saddle-node bifurcation occurs at \( (X_0, \mu_0) \) if:
    \begin{align}
        & \mathbf{W}^T f_{\mu}(X_0, \mu_0) \neq 0, \label{saddle_cond1}\\
        & \mathbf{W}^T [ D_X^2 f(X_0, \mu_0)(\mathbf{V}, \mathbf{V}) ] \neq 0. \label{saddle_cond2}
    \end{align}
\end{itemize}
Here,
 \( f_{\mu} \) denotes the partial derivative of \( f \) with respect to \( \mu \),
 \( D_X f_{\mu} \) denotes the Jacobian of \( f_{\mu} \) with respect to \( X \),
and  \( D_X^2 f \) denotes the second-order derivative of \( f \) with respect to \( X \), applied to two vectors \( (\mathbf{V}, \mathbf{V}) \).
\end{lemma}

In what follows, we analyze codimension-1 bifurcations in system \eqref{equ02} occurring at the boundary equilibria
\begin{equation*} 
E_0 = (0, 0), \quad E_1 = (k_1, 0), \quad \text{and} \quad E_2 = (k_0, 0),
\end{equation*}
as the bifurcation parameter \( \lambda \) varies, while all other parameters remain fixed.

First, we present our main result on transcritical bifurcation.

\begin{theorem} \label{thm4.1}
    The following statements hold regarding transcritical bifurcations in system \eqref{equ02}:
    \begin{itemize}
        \item[(i)] No transcritical bifurcation occurs at the equilibrium $E_0$  as $\lambda$ varies. 
        \item[(ii)] In the case of strong Allee effect, assume that  
        $$
        \begin{cases}
            [sbk_1-2k_0^2r_1(k_1-k_0)]\lambda_0+2Abk_0^2r_1(k_1-k_0)\neq 0,\\
            \lambda_0>0,
        \end{cases}$$
        where $$\lambda_0 = \frac{sb}{k_0(1 - sh)}.$$
        Then, a transcritical bifurcation occurs between the boundary equilibrium $E_2$ and an interior equilibrium
          when  $\lambda = \lambda_0$.
        \item[(iii)] Assume that 
        $1 - sh > 0$, and
                $$
        \begin{cases}
            \frac{sb k_1(b+hk_1\lambda_1)^2 \lambda_1}{r_1k_1(k_0-k_1)}+\lambda_1-Ab \neq 0,\\
            \lambda_1>0,
        \end{cases}$$
       where
       $$\lambda_1 = \frac{sb}{k_1(1 - sh)}.$$

        Then, a transcritical bifurcation occurs between the boundary equilibrium  $E_1$ and an interior equilibrium  when $\lambda = \lambda_1$.
    \end{itemize}
\end{theorem}

\begin{proof}

{\bf Case} {\em (i)}:  
    From Theorems \ref{thm03} and \ref{thm03}, the equilibrium $E_0$ is stable under the strong Allee effect and a saddle point under the weak Allee effect. Consequently, a transcritical bifurcation does not occur at $E_0$. 
 
{\bf Case} {\em (ii)}:  
In the proof of Theorem~\ref{thm03}, we showed that the Jacobian matrix $J(E_2)$ at the equilibrium point $E_2$  has an eigenvalue given by
$$\frac{\lambda k_0}{b+hk_0 \lambda}-s.$$ Therefore, when
$\lambda = \lambda_0,$
this eigenvalue becomes zero. To verify the conditions for a transcritical bifurcation as stated in Lemma~\ref{lem03}, we compute the eigenvectors corresponding to the zero eigenvalue of both the Jacobian $J$ and its transpose $J^T$ at $E_2$:
    \[
    {\bf V} =\begin{pmatrix} v_1 \\ v_2 \end{pmatrix}
     = \begin{pmatrix} \frac{s}{r_1k_0(1 - \frac{k_0}{k_1})} \\ 1 \end{pmatrix}, \quad
    {\bf W} = \begin{pmatrix} 0 \\ 1 \end{pmatrix}.
    \]
 Define the function:
    \[
    {\bf f} =
    \begin{pmatrix}
        f_1 \\ f_2
    \end{pmatrix}
    =
    \begin{pmatrix}
        r_1x(1 - \frac{x}{k_1})(x - k_0) - \frac{(\lambda + Ay)xy}{b + y + hx(\lambda + Ay)} \\
        \frac{(\lambda + Ay)xy}{b + y + hx(\lambda + Ay)} - sy
    \end{pmatrix}.
    \]
The derivative of ${\bf f}$ with respect to $\lambda$ is:
    \[
    {\bf f_{\lambda}} =
     \begin{pmatrix}
        f_{1\lambda} \\ f_{2\lambda} 
    \end{pmatrix}
    =
    \begin{pmatrix}
        -\frac{xy(b + y + hx(\lambda + Ay)) - hx^2y(\lambda + Ay)}{(b + y + hx(\lambda + Ay))^2} \\
        \frac{xy(b + y + hx(\lambda + Ay)) - hx^2y(\lambda + Ay)}{(b + y + hx(\lambda + Ay))^2}
    \end{pmatrix}.
    \]
    Evaluating ${\bf f_{\lambda}}$ at $E_2$ (where $x = k_0$ and $y = 0$) gives ${\bf f_{\lambda}} = (0,0)^T$, satisfying the first condition \eqref{trans_cond1} of Lemma \ref{lem03}.

Next, computing the Jacobian of ${\bf f_{\lambda}}$ at $E_2$:
    \[
    D({\bf f_{\lambda}}(E_2)) =
    \begin{pmatrix}
    \frac{\partial f_{1\lambda}}{\partial x} & \frac{\partial f_{1\lambda}}{\partial y} \\
    \frac{\partial f_{2\lambda}}{\partial x} & \frac{\partial f_{2\lambda}}{\partial y}
    \end{pmatrix}
    =
    \begin{pmatrix}
        0 & \frac{-k_0 b}{(b + y + hx(\lambda + Ay))^4} \\
        0 & \frac{k_0 b}{(b + y + hx(\lambda + Ay))^4}
    \end{pmatrix}.
    \]
This leads to
\begin{align*}
    {\bf W}^T[D{\bf f_{\lambda}}(E_2){\bf V}] & = (0, 1) \Big[ \begin{pmatrix}
        0 & \frac{-k_0 b}{(b+hk_0\lambda)^4} \\
        0 & \frac{k_0 b}{(b+hk_0\lambda)^4}
    \end{pmatrix} \begin{pmatrix}
        r_1k_0(1-\frac{k_0}{k_1}) \\ 1
    \end{pmatrix} \Big] \\
    & = (0,1) \begin{pmatrix}
          \frac{-k_0 b}{(b+hk_0\lambda)^4} \\
          \frac{k_0 b}{(b+hk_0\lambda)^4}
    \end{pmatrix}
    = \frac{k_0 b}{(b+hk_0\lambda)^4} \neq 0.
\end{align*}
Therefore, the second condition \eqref{trans_cond2} of Lemma \ref{lem03} is satisfied.

Finally, for the third condition \eqref{trans_cond3} of Lemma \ref{lem03}, we calculate
\begin{align*}
    {\bf W}^T[D^2{\bf f}(E_2){(\bf V, V)}] 
    & = (0,1) \begin{pmatrix}
          \frac{\partial^2 f_1}{\partial x^2}v_1^2 + \frac{\partial^2 f_1}{\partial x \partial y}v_1v_2 + \frac{\partial^2 f_1}{\partial y^2}v_2^2 \\
           \frac{\partial^2 f_2}{\partial x^2}v_1^2 + \frac{\partial^2 f_2}{\partial x \partial y}v_1v_2 + \frac{\partial^2 f_2}{\partial y^2}v_2^2
    \end{pmatrix}\Big |_{E_2}\\
   & = \frac{\partial^2 f_2}{\partial x^2}v_1^2 + \frac{\partial^2 f_2}{\partial x \partial y}v_1v_2 + \frac{\partial^2 f_2}{\partial y^2}v_2^2 \Big |_{E_2}.
\end{align*}

    At $E_2(k_0,0)$:
    \[
    \frac{\partial^2 f_2}{\partial x^2} = 0, \quad
    \frac{\partial^2 f_2}{\partial x \partial y} = \frac{b\lambda}{(b + hk_0\lambda)^2}, \quad
    \frac{\partial^2 f_2}{\partial y^2} = \frac{2k_0(Ab - \lambda)}{(b + hk_0\lambda)^2}.
    \]
Thus, 
    \[
    {\bf W}^T[D^2{\bf f}(E_2)({\bf V}, {\bf V})]= \frac{[sbk_1-2k_0^2r_1(k_1-k_0)]\lambda+2Abk_0^2r_1(k_1-k_0)}{r_1k_0(b+hk_0\lambda)^2(k_1-k_0)},
    \]
    At the critical value $\lambda=\lambda_0,$
    \[
    {\bf W}^T[D^2{\bf f}(E_2)({\bf V}, {\bf V})]=  \frac{[sbk_1-2k_0^2r_1(k_1-k_0)]\lambda_0+2Abk_0^2r_1(k_1-k_0)}{r_1k_0(b+hk_0\lambda_0)^2(k_1-k_0)}\neq 0.
    \]   
    Hence, all conditions of Lemma \ref{lem03} are satisfied, confirming a transcritical bifurcation at $\lambda=\lambda_0$ between $E_2$ and an interior equilibrium.

{\bf Case} {\em (iii)}:   
From Theorems \ref{thm03} and \ref{thm03}, if $\lambda = \lambda_1$, one eigenvalue of the Jacobian matrix at $E_1$ becomes zero. To verify the conditions for a transcritical bifurcation in Lemma \ref{lem03}, we calculate the eigenvectors corresponding to the zero eigenvalue of both the Jacobian $J$ and its transpose $J^T$ at $E_1$:

    \[
    {\bf V_1} =\begin{pmatrix} v_3 \\ v_4 \end{pmatrix}
     = \begin{pmatrix} \frac{s}{r_1k_1(\frac{k_0}{k_1}-1)} \\ 1 \end{pmatrix}, \quad
    {\bf W} = \begin{pmatrix} 0 \\ 1 \end{pmatrix}.
    \]

Evaluating ${\bf f_{\lambda}}$ at $E_1$ (where $x = k_1$ and $y = 0$), we find ${\bf f_{\lambda}} = (0,0)^T$, satisfying the first condition \eqref{trans_cond1} of Lemma \ref{lem03}.

In addition, we have  
\begin{align*}
    {\bf W}^T[D{\bf f_{\lambda}}(E_1){\bf V_1}] & = (0, 1) \Big[ \begin{pmatrix}
        0 & \frac{-k_1 b}{(b+hk_0\lambda)^4} \\
        0 & \frac{k_1 b}{(b+hk_0\lambda)^4}
    \end{pmatrix} \begin{pmatrix}
        \frac{s}{r_1k_1(\frac{k_0}{k_1}-1)} \\ 1
    \end{pmatrix} \Big] \\
    & = (0,1) \begin{pmatrix}
          \frac{-k_1 b}{(b+hk_0\lambda)^4} \\
          \frac{k_1 b}{(b+hk_0\lambda)^4}
    \end{pmatrix}
    = \frac{k_1 b}{(b+hk_0\lambda)^4} \neq 0.
\end{align*}
Therefore, the second condition \eqref{trans_cond2} of Lemma \ref{lem03} is satisfied.

Finally, for the third condition \eqref{trans_cond3} of Lemma \ref{lem03}, we have
\begin{align*}
    {\bf W}^T[D^2{\bf f}(E_1){(\bf V_1, V_1)}] 
    & = (0,1) \begin{pmatrix}
          \frac{\partial^2 f_1}{\partial x^2}v_3^2 + \frac{\partial^2 f_1}{\partial x \partial y}v_3v_4 + \frac{\partial^2 f_1}{\partial y^2}v_4^2 \\
           \frac{\partial^2 f_2}{\partial x^2}v_3^2 + \frac{\partial^2 f_2}{\partial x \partial y}v_3v_4 + \frac{\partial^2 f_2}{\partial y^2}v_4^2
    \end{pmatrix}\Big |_{E_1}\\
   & = \frac{\partial^2 f_2}{\partial x^2}v_3^2 + \frac{\partial^2 f_2}{\partial x \partial y}v_3v_4 + \frac{\partial^2 f_2}{\partial y^2}v_4^2 \Big |_{E_1}.
\end{align*}
At $E_1(k_1,0)$, 
\begin{align*}
    \frac{\partial^2 f_2}{\partial x^2} = 0, \qquad \frac{\partial^2 f_2}{\partial x \partial y} = \frac{b\lambda}{(b+hk_1\lambda)^2}, \qquad \frac{\partial^2 f_2}{\partial y^2} = \frac{\lambda-Ab}{(b+hk_1\lambda )^4}.
\end{align*}
Thus, 
$${\bf W}^T[D^2{\bf f}(E_1){(\bf V_1, V_1)}]=\left[\frac{sb\lambda k_1(b+hk_1\lambda)^2}{r_1k_1(k_0-k_1)}+\lambda-Ab\right] \frac{1}{(b+hk_1\lambda )^4}.$$

Therefore, at the critical value 
$\lambda=\lambda_1,$ 
$${\bf W}^T[D^2{\bf f}(E_1){(\bf V_1, V_1)}] \neq 0.$$

Since all conditions of Lemma \ref{lem03} are satisfied,  a transcritical bifurcation occurs between the boundary equilibrium  $E_1$ and an interior equilibrium  when $\lambda = \lambda_1$.

This completes the proof.
\end{proof}

Let $\lambda_{TB}$ denote the critical value of the parameter $\lambda$ at which a transcritical bifurcation occurs in system \eqref{equ02}. This value can be determined numerically. Figure \ref{figure2} illustrates examples of transcritical bifurcations and provides approximate values for $\lambda_{TB}$.
\begin{figure}[H]
  \begin{center}  
    \includegraphics[scale=0.46]{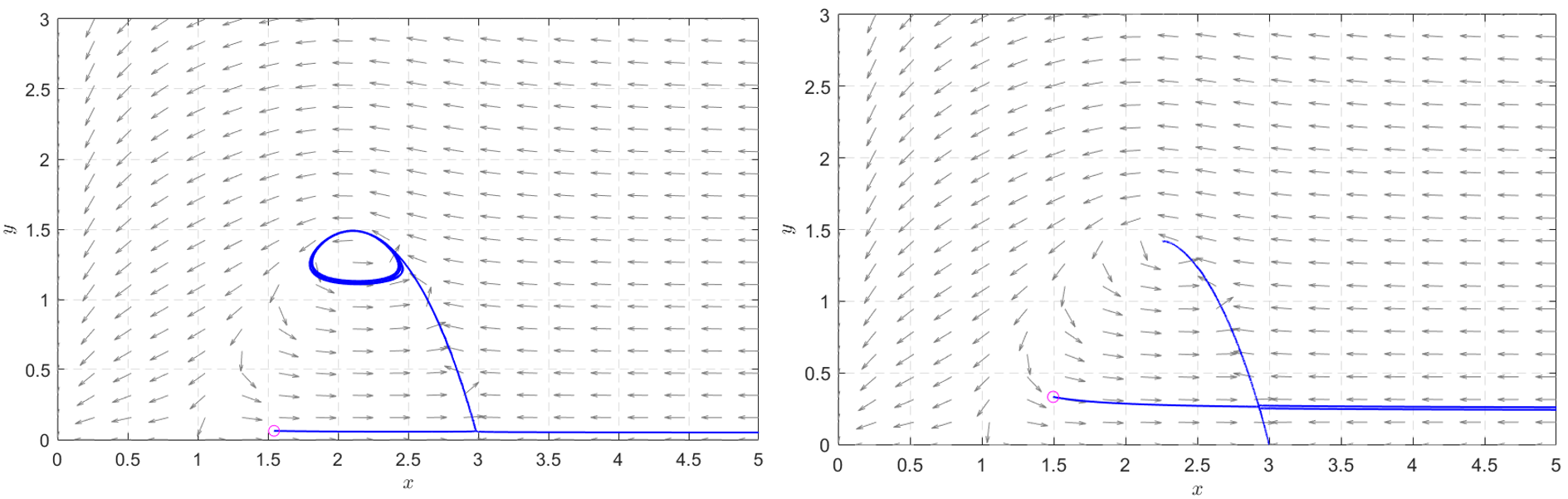}
    \caption{
    Transcritical bifurcation at the equilibrium point $E_1(k_1, 0)$ under the strong Allee effect. The curves represent the predator population ($P$) as a function of the prey population ($N$). Parameter values are set as follows: $r_1 = 1.5$, $k_1 = 3$, $k_0 = 1$, $A = 0.8$, $b = 0.5$, $h = 0.7$, and $s = 0.75$. The estimated critical value of $\lambda$ at the bifurcation point is $\lambda_{TB} =0.263$. The left and right panels correspond to $\lambda = 0.2$ and $\lambda = 0.3$, respectively.
  } \label{figure2}
  \end{center}
\end{figure}

Second, we analyze the saddle-node bifurcation of system~\eqref{equ02}. Let \( \lambda^* > 0 \) be a critical parameter value such that the equation
\begin{equation}\label{G(x)=0}
G(x) = 0,
\end{equation}
has a double positive root \( x^*_{\lambda^*} \in (\max\{0, k_0\}, k_1) \). The function \( G(x) \) is given by
\begin{align}
 G(x) = \frac{r_1}{s} \bigg[
&\frac{A}{k_1} (1 - sh)x^4
 - \left(A(1 - sh) + \frac{s}{k_1} + \frac{Ak_0}{k_1}(sh - 1)\right) x^3 \notag\\
 &+ \left(s + Ak_0(1 - sh) - \frac{sk_0}{k_1} \right)x^2
 - sk_0 x \bigg]
 - \lambda^*(1 - sh)x + sb. \label{gEquation}
\end{align}

Note that equation~\eqref{G(x)=0} is equivalent to
\[
g^{(1)}(x) = g^{(2)}(x),
\]
where \( g^{(1)} \) and \( g^{(2)} \) are defined in \eqref{equ21} and \eqref{equ22}, respectively. In other words, the point \( E^* = (x^*_{\lambda^*}, y^*_{\lambda^*}) \), with \( y^*_{\lambda^*} = g^{(1)}(x^*_{\lambda^*}) > 0 \), is a solution of the algebraic system~\eqref{equ08}, and represents an interior equilibrium of system~\eqref{equ02}.

Figure~\ref{figure22} illustrates the graphs of the nullclines \( f^{(1)}(x, y) = 0 \) and \( f^{(2)}(x, y) = 0 \), where the latter corresponds to \( y = g^{(2)}(x) \). In the left panel, the curves intersect at two points, while in the right panel, they do not intersect. This comparison demonstrates the existence of a critical parameter value \( \lambda^* \) at which the system undergoes a saddle-node bifurcation.

\begin{figure}[H]
  \begin{center}  
    \includegraphics[scale=0.38]{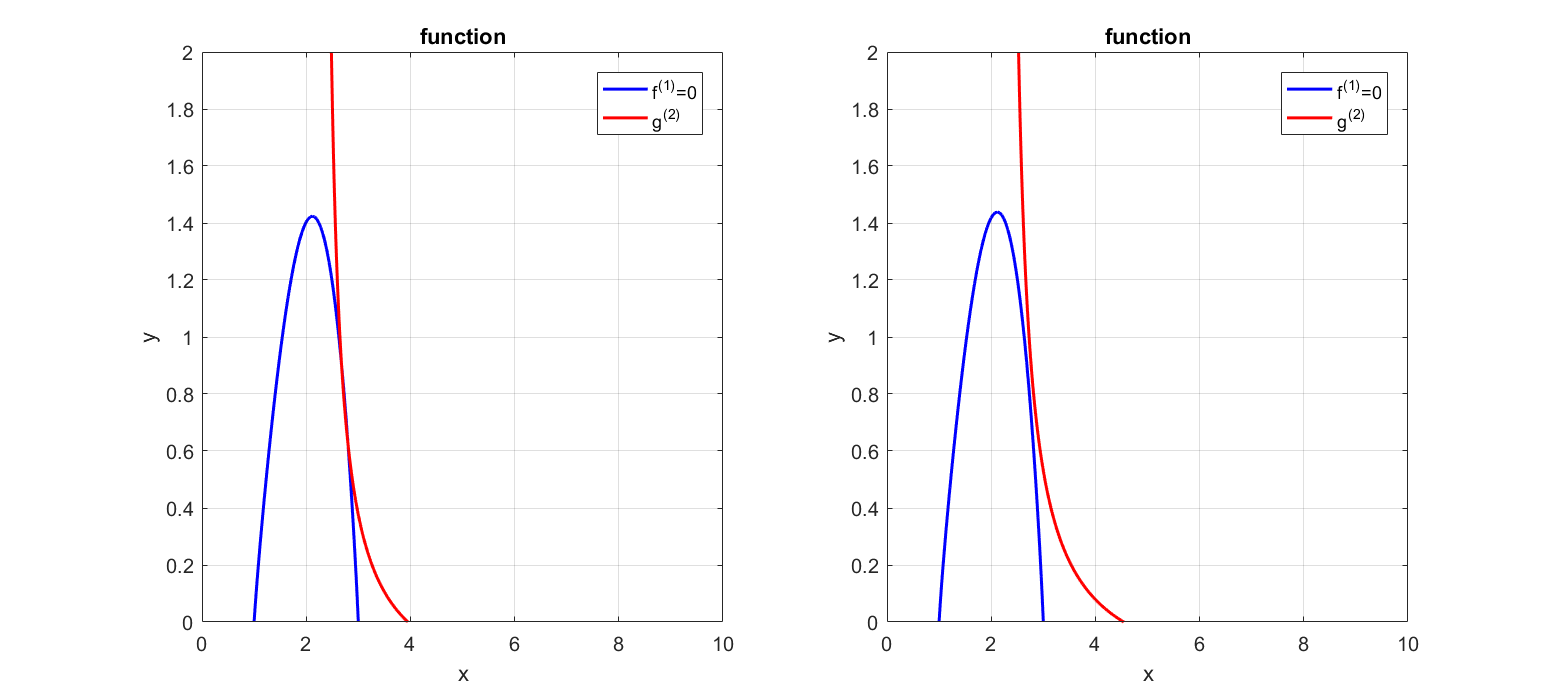}
    \caption{Nullclines of system~\eqref{equ08}: \( f^{(1)}(x, y) = 0 \) and \( f^{(2)}(x, y) = 0 \) (also expressed as \( y = g^{(2)}(x) \)). The left panel shows the curves for \( \lambda = 0.2 \), and the right panel for \( \lambda = 0.23 \). The remaining parameters are fixed at \( r_1 = 1.5 \), \( k_1 = 3 \), \( k_0 = 1 \), \( A = 0.8 \), \( b = 0.5 \), \( h = 0.7 \), and \( s = 0.8 \).
    }
  \label{figure22}
  \end{center}
\end{figure}

\begin{theorem}
Let \( 1 - sh > 0 \). Assume that there exists $\lambda^*>0$ such that equation~\eqref{G(x)=0} has a double positive root $ x^*_{\lambda^*} \in (\max\{0,k_0\},k_1)$, and with \( y^*_{\lambda^*} = g^{(1)}(x^*_{\lambda^*})\), the following conditions hold:
\begin{align}
&r_1(1-\frac{x^*_{\lambda^*}}{k_1})(x^*_{\lambda^*}-k_0) + r_1x^*_{\lambda^*}\left(-\frac{2x^*_{\lambda^*}}{k_1}+1+\frac{k_0}{k_1}\right) \ne 0,  \label{bifurCond1} \\
& \left(-\frac{\frac{\partial f_2}{\partial x}}{\frac{\partial f_1}{\partial x}}-1 \right)\left[\frac{\partial^2 f_2}{\partial x^2}v_5^2+ \frac{\partial^2 f_1}{\partial x \partial y}v_5v_6 + \frac{\partial^2 f_1}{\partial y^2}v_6^2  \right]-\frac{\frac{\partial f_2}{\partial x}}{\frac{\partial f_1}{\partial x}}r_1 \left(2+\frac{2k_0}{k_1}-\frac{6x}{k_1}\right) \Bigg |_{(x^*_{\lambda^*},y^*_{\lambda^*})} \ne 0, \label{bifurCond2}
\end{align}
where the first and second partial derivatives of $f_1, f_2$ are given from 
\eqref{f1x}--\eqref{f2xy} 
and $v_6=1$, 
    $$v_5 = \frac{\lambda x [b+hx(\lambda+Ay)]+2Abxy+{A^2h x^3 y^3}(\lambda+Ay)}{[b+y+hx(\lambda+Ay)]^2[r_1(1-\frac{x}{k_1})(x-k_0)+r_1x(1-\frac{k_0}{k_1}-\frac{2x}{k_1})]-y(b+y)(\lambda+Ay)}.$$ 
 Then, a saddle-node bifurcation of system~\eqref{equ02} occurs at the non-hyperbolic equilibrium point \( E^* = (x^*_{\lambda^*}, y^*_{\lambda^*}) \) for the critical parameter value \( \lambda = \lambda^* \). 
\end{theorem}

\begin{proof}
Since \( 1 - sh > 0 \), Theorem~\ref{two_equilibria} implies that the function $G$ in 
\eqref{gEquation} has at most two  distinct positive roots. 
Furthermore, since $G$ has a double positive root $ x^*_{\lambda^*} \in (\max\{0,k_0\},k_1)$, the two curves 
\( f^{(1)}(x, y) = 0 \) and \( f^{(2)}(x, y) = 0 \) touch each other at $E^*$. Hence, at $E^*$, the total differential is
$$df^{(1)}(x, y)=df^{(2)}(x, y)=0.$$
Thus,
\begin{equation}  \label{estar}
         \frac{\partial f^{(1)}}{\partial x}  \frac{\partial f^{(2)}}{\partial y} = \frac{\partial f^{(1)}}{\partial y}  \frac{\partial f^{(2)}}{\partial x} \quad \text{at } E^*.
    \end{equation}

From \eqref{eqj(e)}, 
 the Jacobian determinant at $E^*$ is
    \[
        \det(J(E^*)) = \frac{\partial f_1}{\partial x}  \frac{\partial f_2}{\partial y} -  \frac{\partial f_1}{\partial y} \frac{\partial f_2}{\partial x}.
    \]
Since $f_1=xf^{(1)}$ and $ f_2=xf^{(2)}$ as defined by \eqref{eq_f1f1},   and $f^{(1)}(E^*)=f^{(2)}(E^*)=0$, it follows from \eqref{estar} that $\det(J(E^*)) =0.$
   Since $J(E^*)$ is nonzero (from \eqref{f2_derivative} and \eqref{f1_derivative}), it has a simple zero eigenvalue.  

Computing the eigenvectors of the zero eigenvalue for $J$ and its transpose $J^T$ at $E^*$, we obtain  
    \[
        {\bf V} = \begin{pmatrix} -\frac{\frac{\partial f_1}{\partial y}}{\frac{\partial f_1}{\partial x}} \\ 1 \end{pmatrix} = \begin{pmatrix} v_5 \\ v_6 \end{pmatrix}, \quad
        {\bf W} = \begin{pmatrix} -\frac{\frac{\partial f_2}{\partial x}}{\frac{\partial f_1}{\partial x}} \\ 1 \end{pmatrix}.
    \]

Since 
    \[
        {\bf f_{\lambda}} =
        \begin{pmatrix}
            f_{1\lambda} \\ f_{2\lambda} 
        \end{pmatrix}
        =
        \begin{pmatrix}
            -\frac{xy[b+y+hx(\lambda+Ay)]-hx^2y(\lambda+Ay)}{(b+y+hx(\lambda + Ay))^2} \\  
            \frac{xy[b+y+hx(\lambda+Ay)]-hx^2y(\lambda+Ay)}{(b+y+hx(\lambda + Ay))^2}
        \end{pmatrix},
    \]
we have
    \begin{align*}
        {\bf W}^T f_{\lambda}(E^*) &= \begin{pmatrix} -\frac{\frac{\partial f_2}{\partial x}}{\frac{\partial f_1}{\partial x}} \\ 1 \end{pmatrix}^T
        \begin{pmatrix}
            -\frac{x^*y^*(b+y^*+hx^*(\lambda+Ay^*))-h(x^*)^2y^*(\lambda+Ay^*)}{(b+y^*+hx^*(\lambda + Ay^*))^2} \\  
            \frac{x^*y^*(b+y^*+hx^*(\lambda+Ay^*))-h(x^*)^2y^*(\lambda+Ay^*)}{(b+y^*+hx^*(\lambda + Ay^*))^2} 
        \end{pmatrix} \\
        &= \frac{x^*y^*(b+y^*)}{(b+y^*+hx^*(\lambda + Ay^*))^2} \left(\frac{\frac{\partial f_2}{\partial x}}{\frac{\partial f_1}{\partial x}} + 1\right) \\
        &= \frac{x^*y^*(b+y^*) \left[r_1(1-\frac{x^*_{\lambda^*}}{k_1})(x^*_{\lambda^*}-k_0) + r_1x^*_{\lambda^*}\left(-\frac{2x^*_{\lambda^*}}{k_1}+1+\frac{k_0}{k_1}\right)\right]}{(b+y^*+hx^*(\lambda + Ay^*))^2 \frac{\partial f_1}{\partial x}} \neq 0
    \end{align*}
    due to \eqref{bifurCond1}. 
    Thus, the condition \eqref{saddle_cond1} of Lemma \ref{lem03} is satisfied.   

Next, we verify condition \eqref{saddle_cond2} of Lemma \ref{lem03} by calculating ${\bf W}^T[D^2f(E^*)({\bf V},{\bf V})]$ as follows:
    \begin{align*}
        {\bf W}^T[D^2{\bf f}(E^*)({\bf V}, {\bf V})] 
        &= \begin{pmatrix} -\frac{\frac{\partial f_2}{\partial x}}{\frac{\partial f_1}{\partial x}} \\ 1 \end{pmatrix}^T 
        \begin{pmatrix}
              \frac{\partial^2 f_1}{\partial x^2}v_5^2 + \frac{\partial^2 f_1}{\partial x \partial y}v_5v_6 + \frac{\partial^2 f_1}{\partial y^2}v_6^2 \\
              \frac{\partial^2 f_2}{\partial x^2}v_5^2 + \frac{\partial^2 f_2}{\partial x \partial y}v_5v_6 + \frac{\partial^2 f_2}{\partial y^2}v_6^2
        \end{pmatrix}.
    \end{align*}
Hence, we obtain that  
\begin{align*}
    &{\bf W}^T[D^2f(E^*)({\bf V},{\bf V})] \\
    &=\left(-\frac{\frac{\partial f_2}{\partial x}}{\frac{\partial f_1}{\partial x}}-1 \right)\left[\frac{\partial^2 f_2}{\partial x^2}v_5^2+ \frac{\partial^2 f_1}{\partial x \partial y}v_5v_6 + \frac{\partial^2 f_1}{\partial y^2}v_6^2  \right]-\frac{\frac{\partial f_2}{\partial x}}{\frac{\partial f_1}{\partial x}}r_1 \left(2+\frac{2k_0}{k_1}-\frac{6x}{k_1}\right) \Bigg |_{(x^*_{\lambda^*},y^*_{\lambda^*})}\\
    &\ne 0
\end{align*}
due to \eqref{bifurCond2}.
This implies the condition \eqref{saddle_cond2}. By Lemma \ref{lem03}, a saddle-node bifurcation occurs at $E^*$. The proof is complete.
\end{proof}

\begin{remark}
    A necessary condition for the assumption that the equation \eqref{G(x)=0} has a double positive root 
    $ x^*_{\lambda^*} \in  (\max\{0,k_0\},k_1)$
    is that  $\lambda^*-Ab <0$.
\end{remark}

Figure \ref{figure3} illustrates an example of saddle-node bifurcation  in system \eqref{equ02} with parameter values $r_1 = 1.5$, $k_1 = 3$, $k_0 = 1$, $A = 0.8$, $b = 0.5$, $h = 0.7,$ and $s = 0.8$.   The left sub-figure shows the absence of interior equilibria for $\lambda = 0.23,$ while the right sub-figure depicts the emergence of two interior equilibria for
$\lambda = 0.2$.
\begin{figure}[H]
  \begin{center}  
    \includegraphics[scale=0.48]{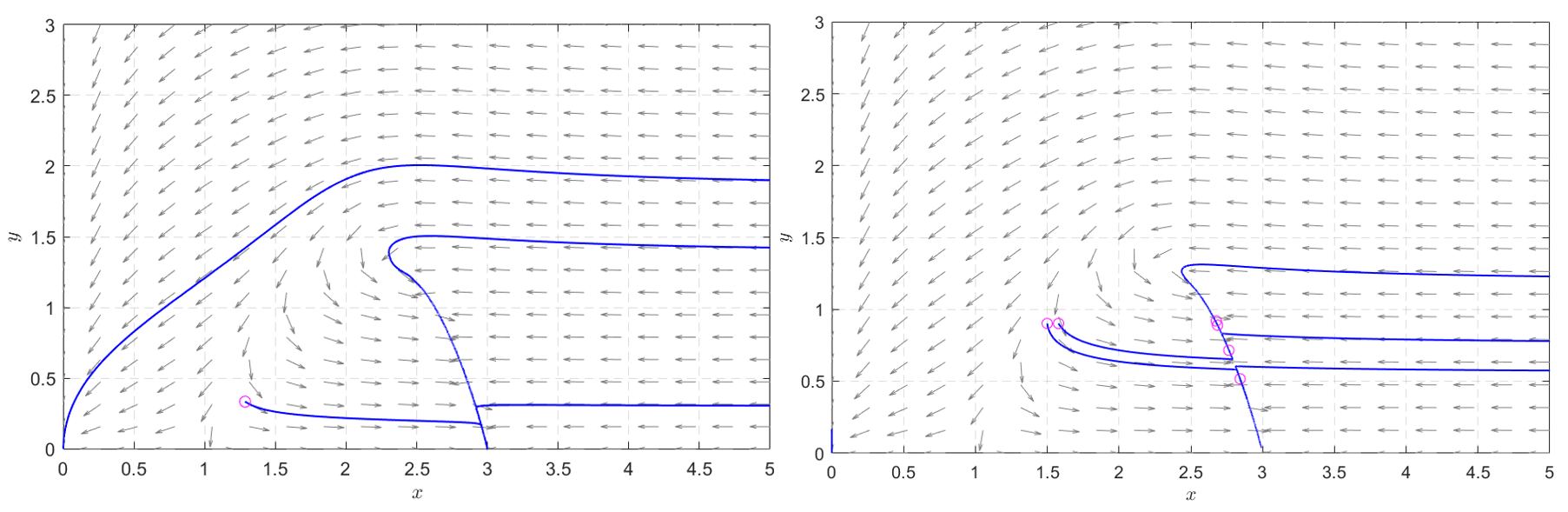}
    \caption{Saddle-node bifurcation at an interior equilibrium under the strong Allee effect. 
    The curves illustrate the predator population ($y$) versus the prey population ($x$). 
    The left panel corresponds to $\lambda = 0.23$, showing no interior equilibria, 
    while the right panel corresponds to $\lambda = 0.2$, demonstrating the emergence of two interior equilibria.
 }  \label{figure3}
  \end{center}
\end{figure}

\subsection{Hopf bifurcation}
In this subsection, we investigate conditions under which a Hopf bifurcation occurs in system \eqref{equ02}. A Hopf bifurcation is a local bifurcation in which a pair of complex conjugate eigenvalues of the Jacobian matrix at an equilibrium point crosses the imaginary axis as a system parameter varies. This transition leads to a qualitative change in the local dynamics: the equilibrium may lose stability, and a family of small-amplitude periodic orbits (limit cycles) may emerge or disappear.

More formally, consider the parameter-dependent system
\[
\dot{z} = f(z, \kappa), \quad z \in \mathbb{R}^n, \ \kappa \in \mathbb{R},
\]
where $f \in C^r(\mathbb{R}^n \times \mathbb{R}, \mathbb{R}^n)$ with $r \geq 2$. Suppose that:
\begin{itemize}
    \item[(H1)] The system has a smooth branch of equilibria \( z_0(\kappa) \), such that \( f(z_0(\kappa), \kappa) = 0 \) for \( \kappa \) near \( \kappa_0 \).
    \item[(H2)] The Jacobian matrix \( Df(z_0(\kappa_0), \kappa_0) \) has a simple pair of purely imaginary eigenvalues \( \pm i \omega_0 \) with \( \omega_0 > 0 \), and all other eigenvalues have nonzero real parts.
    \item[(H3)] The real part of the eigenvalues crosses the imaginary axis transversally:
    \[
    \left. \frac{d}{d\kappa} \operatorname{Re}(\lambda(\kappa)) \right|_{\kappa = \kappa_0} \neq 0.
    \]
\end{itemize}
Then, the system undergoes a Hopf bifurcation at \( (z_0(\kappa_0), \kappa_0) \); that is, a family of periodic solutions bifurcates from the equilibrium. The direction and stability of these periodic solutions are determined by the sign of the first Lyapunov coefficient $l_1$. A supercritical Hopf bifurcation ($l_1 < 0$) yields stable limit cycles, while a subcritical one ($l_1 > 0$) yields unstable limit cycles (see \cite{kuznetsov1998elements}).

We now state the main result for system \eqref{equ02}.

\begin{theorem}\label{hopfBifurcation}
Consider system \eqref{equ02}.
\begin{itemize}
    \item[(i)] Suppose \( \lambda - Ab < 0 \). Asume that the system has two distinct interior equilibria, denoted as \( E_1^* \) and \( E_2^* \), in \( \mathbb{R}_+^2 \), 
            where the \( x \)-coordinate of \( E_1^* \) is smaller than that of \( E_2^* \),
    or the system has only one interior equilibrium $E_1^*$.
    Assume that the Jacobian matrix \( J(E_1^*) \) satisfies \( \operatorname{Tr}(J(E_1^*)) = 0 \) when \( s=s^*\), then the system undergoes a Hopf bifurcation at \( (E_1^*,s^*) \).
    
    \item[(ii)] Suppose \( \lambda - Ab > 0 \) and system \eqref{equ02} admits two interior equilibria, \( E_1^* \) and \( E_2^* \), with the $x$-coordinate of \( E_1^* \) smaller than that of \( E_2^* \). If \( \operatorname{Tr}(J(E_1^*)) = 0 \) and \( \operatorname{Tr}(J(E_2^*)) = 0 \) when \( s=s^*\), then the system undergoes a Hopf bifurcation at \( E_2^* \) under a strong Allee effect.
    
    \item[(iii)] Suppose \( \lambda - Ab > 0 \) and system \eqref{equ02} admits a unique interior equilibrium \( E_1^* \), for which \( \operatorname{Tr}(J(E_1^*)) = 0 \) at when \( s=s^*\). Then, the system undergoes a Hopf bifurcation at \( E_1^* \).
\end{itemize}
\end{theorem}

\begin{proof}
    We prove the theorem for the case \( \lambda - Ab < 0 \). The proofs for the remaining cases follow similar arguments and are therefore omitted.

    Under the condition \( \lambda - Ab < 0 \), we showed in the proof of Theorem \ref{thm01} that \( \det(J(E^*)) > 0 \). Combining this and the condition \( \operatorname{Tr}(J(E_1^*)) = 0 \), we obtain that 
   the eigenvalues of $J(E_1^*)$ are purely imaginary.


    Furthermore, we compute the derivative of the trace of the Jacobian with respect to \( s \). From \eqref{eqj(e)}, we have  
    \begin{align*}
        \frac{\partial \operatorname{Tr}(J(E^*))}{\partial s} 
        &= \frac{\partial}{\partial s} \left( \frac{\partial f_1}{\partial x} \right) 
        + \frac{\partial}{\partial s} \left( \frac{\partial f_2}{\partial y} \right) = -1 \neq 0.
    \end{align*}
    Since this derivative is nonzero, the transversality condition for a Hopf bifurcation is satisfied. Therefore, a Hopf bifurcation occurs at \( E^* \) as the parameter \( s \) varies.
\end{proof}

In Figure \ref{figure4}, we illustrate a Hopf bifurcation in system \eqref{equ02} under a strong Allee effect for the following parameter values: \( r_1 = 1.5 \), \( k_1 = 3 \), \( k_0 = 1 \), \( A = 0.8 \), \( b = 0.5 \), \( h = 0.7 \), and \( \lambda = 0.3 \). The left sub-figure shows the presence of an interior equilibrium for \( s = 0.75 \), while the right sub-figure depicts the emergence of a limit cycle for \( s = 0.76 \).
\begin{figure}[H]
  \begin{center}  
    \includegraphics[scale=0.48]{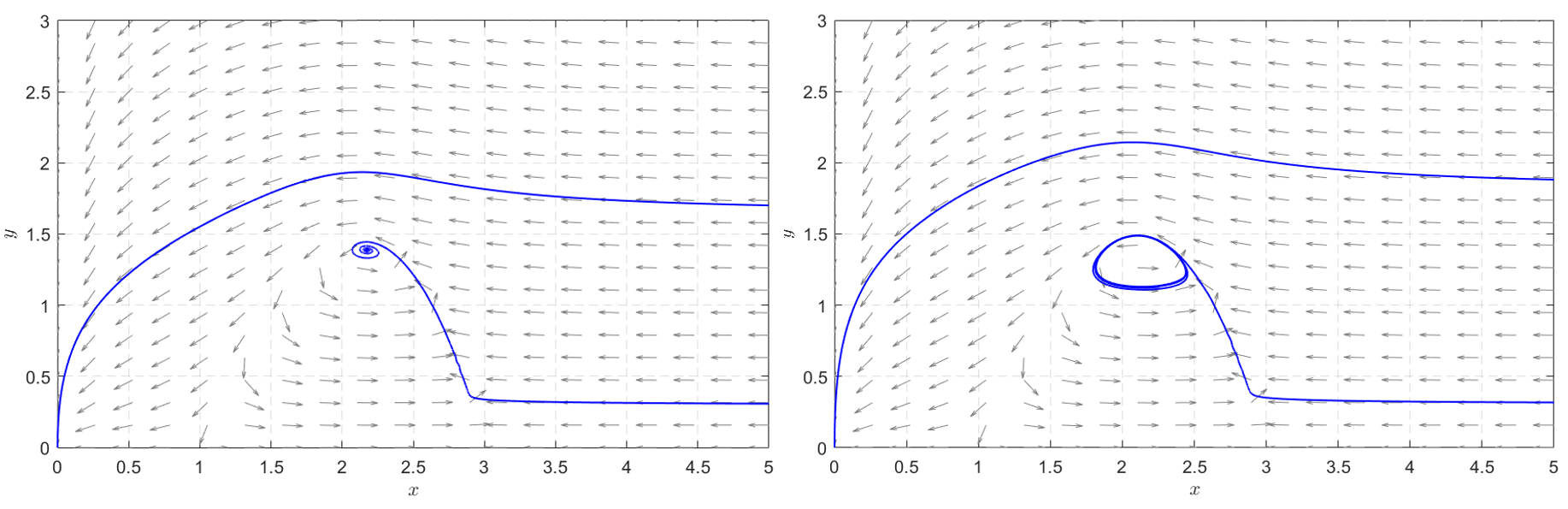}
    \caption{Hopf bifurcation at an interior equilibrium under the strong Allee effect. The curves illustrate the predator population (\( y \)) versus the prey population (\( x \)). The left panel shows stable equilibrium points for \( s = 0.75 \), while the right panel depicts a stable limit cycle for \( s = 0.76 \).}  \label{figure4}
  \end{center}
\end{figure}
Similarly, Figure \ref{figure5} illustrates a Hopf bifurcation in system \eqref{equ02} under a weak Allee effect. Here, the parameter values are \( r_1 = 0.3 \), \( k_1 = 4 \), \( k_0 = 1 \), \( A = 0.1 \), \( b = 3.5 \), \( h = 0.9 \), and \( \lambda = 0.1 \), with \( s = 0.1 \) and \( s = 0.15 \), respectively.
\begin{figure}[H]
  \begin{center}  
    \includegraphics[scale=0.52]{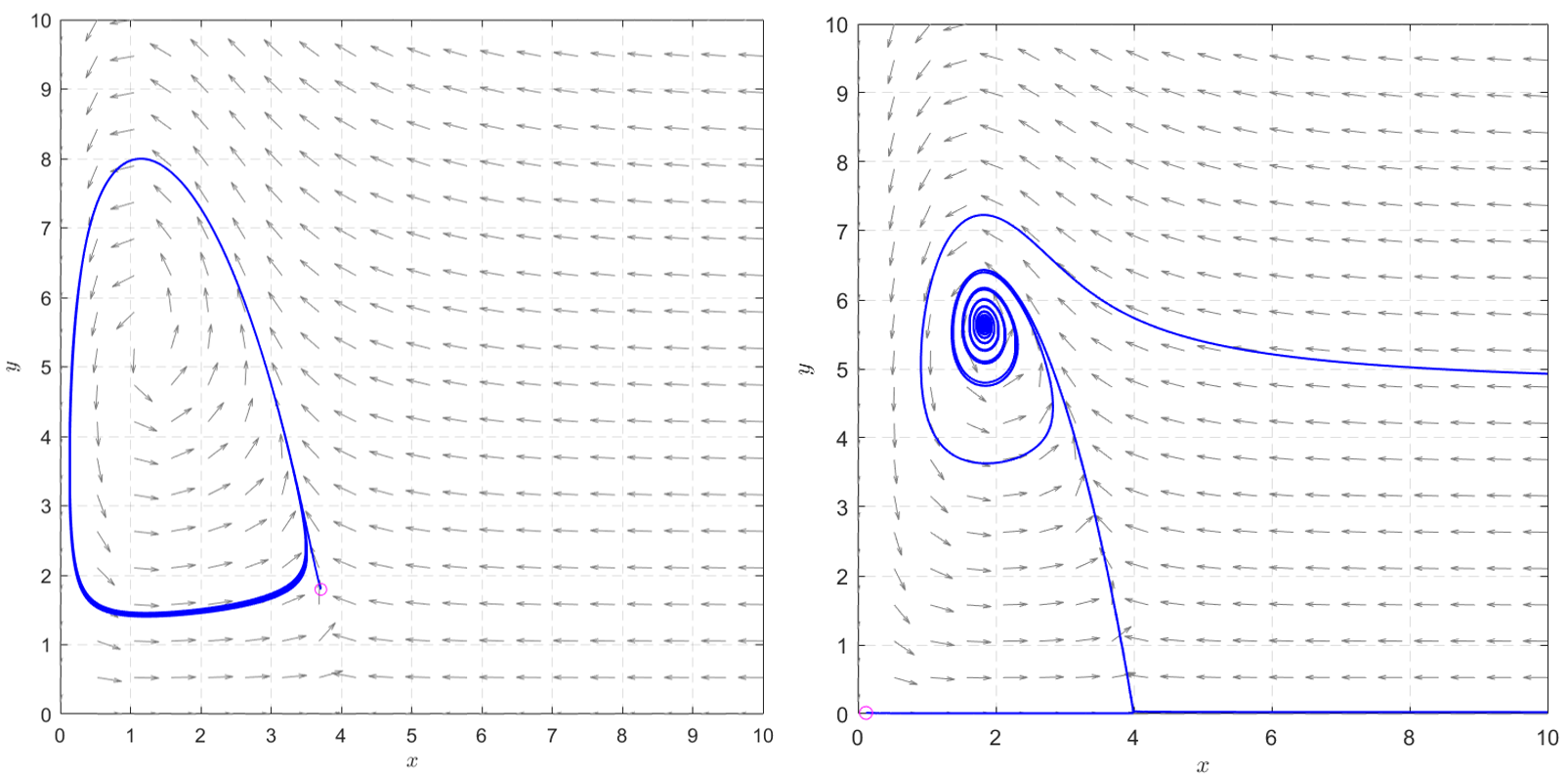}
    \caption{Hopf bifurcation at an interior equilibrium under the weak Allee effect. The curves illustrate the predator population (\( y \)) versus the prey (\( x \)) population. The left panel shows stable equilibrium points for \( s = 0.1 \), while the right panel depicts a stable limit cycle for \( s = 0.15 \). }  \label{figure5}
  \end{center}
\end{figure}

\subsection{Heteroclinic orbits and bifurcation}
A heteroclinic orbit (or heteroclinic connection)  is a trajectory in phase space that approaches two distinct equilibrium points as time tends to both positive and negative infinity. A heteroclinic bifurcation occurs when there is 
a qualitative change in a heteroclinic orbit as a system parameter varies, often involving the creation or destruction of such connections.

\begin{theorem} \label{Heteroclinic_orbit}  
Consider the case where the system~\eqref{equ02} exhibits a strong Allee effect.  
If \( 1 - sh > 0 \) and the parameter \( \lambda \) lies in the interval
\[
\lambda \in \left( \frac{sb}{k_1(1-sh)}, \frac{sb}{k_0(1-sh)} \right),
\]
then the system possesses a stable manifold \( \Gamma_1 \) associated with the saddle point \( E_2 = (k_0, 0) \), and an unstable manifold \( \Gamma_2 \) associated with the saddle point \( E_1 = (k_1, 0) \).  
Moreover, there exists a critical value \( \tilde{\lambda} \in \left( \frac{sb}{k_1(1-sh)}, \frac{sb}{k_0(1-sh)} \right) \) such that the manifolds \( \Gamma_1 \) and \( \Gamma_2 \) coincide, forming a heteroclinic orbit.  
In other words, the system~\eqref{equ02} undergoes a heteroclinic bifurcation at \( \lambda = \tilde{\lambda} \).
\end{theorem}

\begin{proof}
Let \( \lambda \in \left( \frac{sb}{k_1(1-sh)}, \frac{sb}{k_0(1-sh)} \right) \).  
Under this assumption, we have \( g^{(2)}(k_0) < 0 \) and \( g^{(2)}(k_1) > 0 \).  
By continuity of \( g^{(2)} \) and Theorem~\ref{two_equilibria}, the system admits a unique interior equilibrium.

Moreover, since
\[
\frac{\lambda k_1}{b + hk_1 \lambda} - s > 0 \quad \text{and} \quad \frac{\lambda k_0}{b + hk_0 \lambda} - s < 0,
\]
Theorem~\ref{thm03} implies that both \( E_1 \) and \( E_2 \) are hyperbolic saddle points. Let \( \Gamma_1 \) denote the stable manifold of \( E_2 \), and \( \Gamma_2 \) the unstable manifold of \( E_1 \).

Let \( \lambda_2 < 0 \) be the negative eigenvalue of the Jacobian \( J(E_2) \). The matrix \( \lambda_2 I - J(E_2) \) is:
\[
\begin{pmatrix}
\lambda_2 - r_1 k_0 \left(1 - \frac{k_0}{k_1} \right) & \frac{\lambda k_0}{b + h\lambda k_0} \\
0 & 0
\end{pmatrix}.
\]
The corresponding eigenvector is
\[
\mathbf{v} =
\begin{pmatrix}
- \frac{\lambda k_0}{(b + h\lambda k_0)\left[\lambda_2 - r_1 k_0 \left(1 - \frac{k_0}{k_1} \right)\right]} \\
1
\end{pmatrix}.
\]
Since \( E_2 \) is a hyperbolic saddle, the Stable Manifold Theorem ensures that \( \Gamma_1 \) is a smooth curve tangent to the stable eigenspace. Therefore, its slope at \( E_2 \) is:
\[
\alpha_1 = \frac{1}{-\frac{\lambda k_0}{(b + h\lambda k_0)\left[\lambda_2 - r_1k_0(1 - \frac{k_0}{k_1})\right]}}= \frac{\left[-\lambda_2 + r_1 k_0 \left(1 - \frac{k_0}{k_1} \right)\right](b + h\lambda k_0)}{\lambda k_0} > 0.
\]

We now compare this with the slope of the nullcline \( f^{(1)}(x,y) = 0 \) at \( E_2 \):
\[
\alpha_2 = \left. -\frac{\partial f^{(1)}/\partial x}{\partial f^{(1)}/\partial y} \right|_{E_2} = \frac{r_1\left(1 - \frac{k_0}{k_1} \right)(b + h k_0 \lambda)}{\lambda} > 0.
\]
Since \( -\lambda_2 > 0 \), it follows that \( \alpha_1 > \alpha_2 \), meaning \( \Gamma_1 \) lies above the nullcline \( f^{(1)} = 0 \) in a neighborhood of \( E_2 \).

Define the following regions:
\[
C := \{(x, y) \in \mathbb{R}_+^2 : f^{(1)}(x, y) < 0 \}, \quad
D := \{(x, y) \in \mathbb{R}_+^2 : f^{(2)}(x, y) < 0 \}.
\]

We claim that \( \Gamma_1 \cap \{ f^{(1)} = 0 \} \cap \overline{C \cap D} = \{E_2\} \).  
Suppose, for contradiction, that another intersection point \( (\bar{x}, \bar{y}) \neq E_2 \) exists in \( \overline{C \cap D} \).  
Then the trajectory \( (x_{\Gamma_1}(t), y_{\Gamma_1}(t)) \) starting at \( (\bar{x}, \bar{y}) \in \Gamma_1 \) satisfies:
\begin{align*}
\lim_{t \to 0} \frac{d y_{\Gamma_1}(t)}{d x_{\Gamma_1}(t)} 
&= \lim_{t \to 0} 
\frac{
\displaystyle \frac{(\lambda + A y_{\Gamma_1}) x_{\Gamma_1} y_{\Gamma_1}}{b + y_{\Gamma_1} + h x_{\Gamma_1} (\lambda + A y_{\Gamma_1})} - s y_{\Gamma_1}
}{
\displaystyle r_1 x_{\Gamma_1} \left(1 - \frac{x_{\Gamma_1}}{k_1} \right) (x_{\Gamma_1} - k_0) - \frac{(\lambda + A y_{\Gamma_1}) x_{\Gamma_1} y_{\Gamma_1}}{b + y_{\Gamma_1} + h x_{\Gamma_1} (\lambda + A y_{\Gamma_1})}
} \\
&= \lim_{t \to 0} \frac{f^{(2)}(x_{\Gamma_1}(t), y_{\Gamma_1}(t))  y_{\Gamma_1}(t)}{f^{(1)}(x_{\Gamma_1}(t), y_{\Gamma_1}(t))  x_{\Gamma_1}(t)} = \infty,
\end{align*}
since \( f^{(1)}(\bar{x}, \bar{y}) = 0 \).  
This contradicts the fact that \( \Gamma_1 \) is a smooth curve: the trajectory \( (x_{\Gamma_1}(t), y_{\Gamma_1}(t)) \) converges to \( E_2 \) as \( t \to \infty \), and therefore its slope at \( (\bar{x}, \bar{y}) \) must be finite. Hence, \( \Gamma_1 \) lies strictly above the nullcline \( f^{(1)} = 0 \) in \( C \cap D \), except at \( E_2 \).

Next, we show that \( \Gamma_1 \) intersects the curve \( f^{(2)}(x, y) = 0 \).  
Suppose, for contradiction, that no such intersection exists.  Let \( E_{\Gamma_1} \in \Gamma_1 \cap (C \cap D) \), and consider the backward trajectory \( (x^-_{\Gamma_1}(t), y^-_{\Gamma_1}(t)) \) with \( t \leq 0 \) starting from \( E_{\Gamma_1} \).  
As \( t \to -\infty \), this trajectory approaches:
\[
\lim_{t \to -\infty} x^-_{\Gamma_1}(t) = \frac{s}{A(1-sh)}, \quad \lim_{t \to -\infty} y^-_{\Gamma_1}(t) = \infty.
\]

Fix \( \epsilon > 0 \). Then there exists \( T < 0 \) such that
\[
x^-_{\Gamma_1}(t) \in \left[\frac{s}{A(1-sh)} - \epsilon, \frac{s}{A(1-sh)}\right) \quad \text{for all } t \leq T.
\]
Due to continuity of \( f^{(1)} \), there exists a constant \( K_1(\epsilon) > 0 \) such that
\[
f^{(1)}(x^-_{\Gamma_1}(t), y^-_{\Gamma_1}(t)) \leq -K_1(\epsilon) \quad \text{for all } t \leq T.
\]
Additionally, by Theorem~\ref{thm02}, trajectories remain bounded, so
\[
\left| f^{(2)}(x^-_{\Gamma_1}(t), y^-_{\Gamma_1}(t)) y^-_{\Gamma_1}(t) \right| \leq K \quad \text{for all } t \leq 0.
\]
Combining the above gives
\[
\left| \frac{d y^-_{\Gamma_1}(t)}{d x^-_{\Gamma_1}(t)} \right| 
= \left| \frac{f^{(2)}(x^-_{\Gamma_1}(t), y^-_{\Gamma_1}(t)) y^-_{\Gamma_1}(t)}{f^{(1)}(x^-_{\Gamma_1}(t), y^-_{\Gamma_1}(t))  x^-_{\Gamma_1}(t)} \right| 
\leq \left| \frac{K}{x^-_{\Gamma_1}(t) K_1(\epsilon)} \right| 
\leq \left| \frac{K}{\left[\frac{s}{A(1-sh)} - \epsilon\right] K_1(\epsilon)} \right|.
\]
contradicting the previous result that \( \Gamma_1(x) \to \infty \) as \( x \to \left(\frac{s}{A(1-sh)}\right)^- \).  
Hence, \( \Gamma_1 \) must intersect \( f^{(2)} = 0 \); denote this point by \( (x_1, y_1) \).  
Likewise, let \( (x_2, y_2) \) denote the intersection of \( \Gamma_2 \) with \( f^{(2)} = 0 \).

We now show that there exists a critical value \( \tilde{\lambda} \) such that 
\[
(x_1, y_1) = (x_2, y_2).
\]

Indeed, let \( (x^*, y^*) \) be the intersection point of the nullclines \( f^{(1)} = 0 \) and \( f^{(2)} = 0 \). The value \( y_1 \) varies continuously with \( \lambda \) due to the continuity of both \( \Gamma_1 \) and \( f^{(1)} = 0 \). Similarly, \( y_2 \) also depends continuously on \( \lambda \). Thus, there exists a small constant \( m > 0 \) such that:
\begin{itemize}
    \item If \( \lambda \in \left[\frac{sb}{k_1(1 - sh)} + m,\ \frac{sb}{k_0(1 - sh)} - m \right] \), then both \( y_1 \) and \( y_2 \) vary continuously with \( \lambda \).
    
    \item When \( \lambda = \frac{sb}{k_1(1 - sh)} + m \), we have \( y_1 > y_2 \) since, as \( \lambda \to \frac{sb}{k_1(1 - sh)} \), the point \( (x_2, y_2) \) tends to \( E_1 = (k_1, 0) \), and \( \Gamma_1 \) is increasing with respect to \( x \) in \( C \cap D \).
    
    \item When \( \lambda = \frac{sb}{k_0(1 - sh)} - m \), we have \( y_1 < y_2 \) for similar reasons.
\end{itemize}

By the Intermediate Value Theorem, there exists a critical value \( \tilde{\lambda} \in \left( \frac{sb}{k_1(1-sh)}, \frac{sb}{k_0(1-sh)} \right) \) such that \( y_1 = y_2 \), and thus \( (x_1, y_1) = (x_2, y_2) \).  
 The uniqueness of the solution to system~\eqref{equ02} implies that this common point corresponds to a heteroclinic orbit connecting the saddle points \( E_1 \) and \( E_2 \). This completes the proof.
\end{proof}

In Figure~\ref{figure6}, we illustrate the existence of heteroclinic orbits in system~\eqref{equ02} for two different values of the parameter \( \lambda \). The parameter values are chosen as follows: \( r_1 = 0.6 \), \( k_1 = 5 \), \( k_0 = 0.5 \), \( A = 0.1 \), \( b = 1.5 \), \( h = 0.9 \), and \( s = 0.44 \). 

In the left panel, corresponding to \( \lambda = 0.9776 \), the stable manifold \( \Gamma_1 \) lies above the unstable manifold \( \Gamma_2 \). In contrast, the right panel shows the case \( \lambda = 0.9773 \), where \( \Gamma_1 \) lies below \( \Gamma_2 \). These observations indicate that there exists a critical value \( \tilde{\lambda} \in (0.9773, 0.9776) \) at which the manifolds coincide, confirming the presence of a heteroclinic orbit in the system.

\begin{figure}[H]
  \begin{center}  
    \includegraphics[scale=0.47]{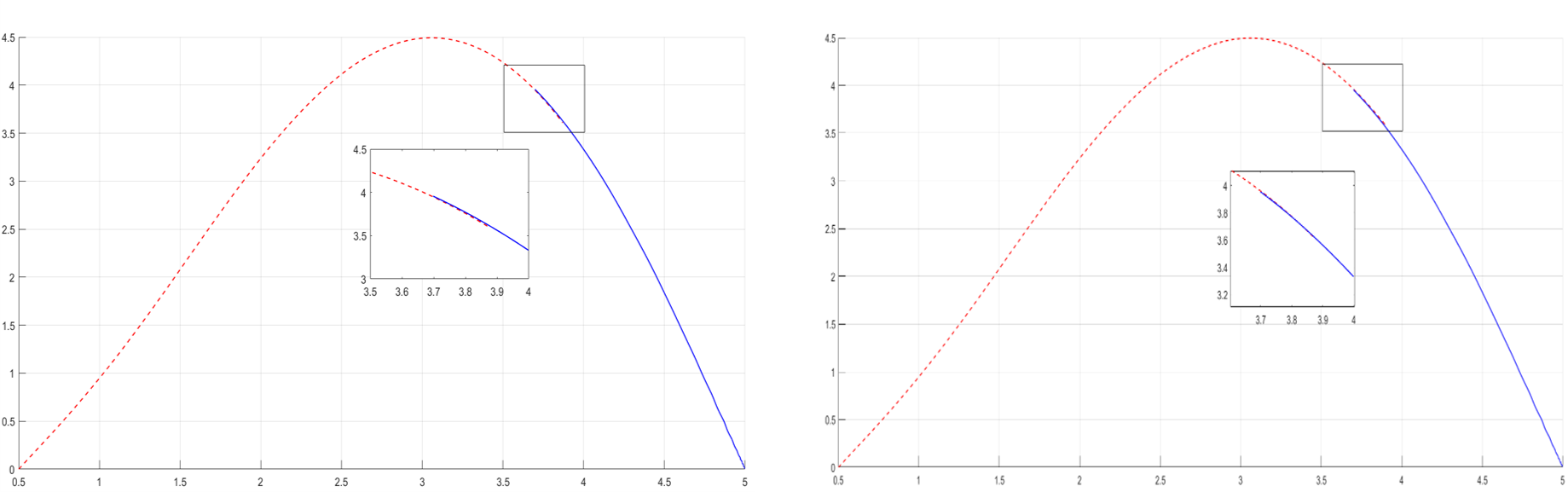}
    \caption{Existence of heteroclinic orbits in system~\eqref{equ02}. The dotted curve represents the stable manifold \( \Gamma_1 \), while the solid curve corresponds to the unstable manifold \( \Gamma_2 \). Left panel: for \( \lambda = 0.9776 \), \( \Gamma_1 \) lies above \( \Gamma_2 \). Right panel: for \( \lambda = 0.9773 \), \( \Gamma_1 \) lies below \( \Gamma_2 \). These configurations imply that a heteroclinic orbit exists for some \( \lambda \in (0.9773, 0.9776) \).}
    \label{figure6}
  \end{center}
\end{figure}

Figure~\ref{figure7} illustrates a heteroclinic bifurcation in system~\eqref{equ02} with parameter values \( r_1 = 0.6 \), \( k_1 = 5 \), \( k_0 = 0.5 \), \( A = 0.1 \), \( b = 1.5 \), \( h = 0.9 \), and \( s = 0.44 \). The left and right panels display the manifold originating from \( E_1 \) for \( \lambda = 0.9773 \) and \( \lambda = 0.9776 \), respectively. 

As seen in the figure, the manifold structure changes as \( \lambda \) varies, indicating that a closed orbit forms at some critical value \( \tilde{\lambda} \in (0.9773, 0.9776) \), consistent with the occurrence of a heteroclinic bifurcation.

\begin{figure}[H]
  \begin{center}  
    \includegraphics[scale=0.53]{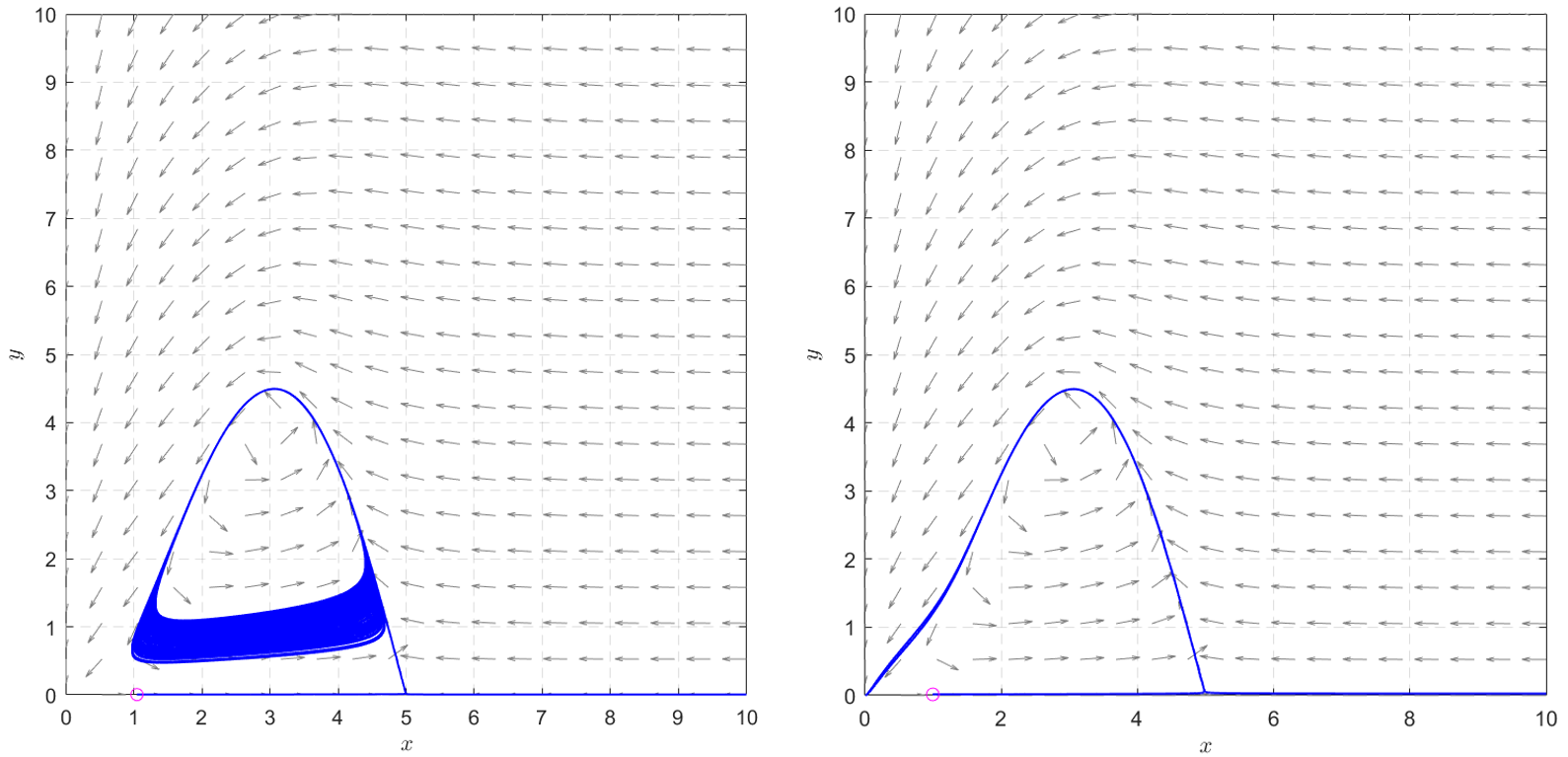}
    \caption{Heteroclinic bifurcation in system~\eqref{equ02}. The left and right panels show the manifold emanating from \( E_1 \) for \( \lambda = 0.9773 \) and \( \lambda = 0.9776 \), respectively. The qualitative change in manifold behavior supports the existence of a heteroclinic orbit at some intermediate value \( \tilde{\lambda} \in (0.9773, 0.9776) \).}
    \label{figure7}
  \end{center}
\end{figure}

\subsection{Global transcritical bifurcation on a closed cycle}
In this subsection, we explore global bifurcations, such as heteroclinic  and transcritical bifurcations on invariant cycle, which involve topological changes in the system's phase space.

\begin{theorem}
    If the parameter $k_0$ transitions from positive to negative, specifically $k_0 = 0$, the system \eqref{equ02} undergoes a transcritical bifurcation at $E_0$. Furthermore, if $k_0$ becomes negative and the system lacks both stable interior equilibria and periodic solutions, a closed cycle emerges.
\end{theorem}
\begin{proof}
    When $k_0 = 0$ and $(x,y) = (0, 0)$, the Jacobian matrix becomes
    \begin{align*}
      J = 
        \begin{pmatrix}
            0 & 0 \\
            0 & -s
        \end{pmatrix}.
    \end{align*}
  The corresponding eigenvectors are ${\bf V} = (1, 0)^T$, and ${\bf W }= (1, 0)^T$.
  Defining
  \begin{align*}
      {\bf f_{k_0}} = 
        \begin{pmatrix}
             -r_1x(1-\frac{x}{k_1}) \\
            0 
        \end{pmatrix},
    \end{align*}
    we find ${\bf f_{k_0}(E_0)} = (0, 0)^T$, satisfying the condition \eqref{trans_cond1} of Lemma \ref{lem03}. The conditions \eqref{trans_cond2} and \eqref{trans_cond3}  are also verified:
  \begin{align*}
    {\bf W}^T[D{\bf f_{k_0}}(E_0){\bf V}] & = (1, 0) \big[ \begin{pmatrix}
        -r_1 & 0 \\
        0 & 0
    \end{pmatrix} \begin{pmatrix}
       1 \\ 0
    \end{pmatrix} \big] 
     \neq 0,
\end{align*}
and  
\begin{align*}
    {\bf W}^T[D^2{\bf f}(E_0){(\bf V, V)}] 
    & = (1, 0) \begin{pmatrix}
          \frac{\partial^2 f_1}{\partial x^2}v_1^2 + \frac{\partial^2 f_1}{\partial x \partial y}v_1v_2 + \frac{\partial^2 f_1}{\partial y^2}v_2^2 \\
           \frac{\partial^2 f_2}{\partial x^2}v_1^2 + \frac{\partial^2 f_2}{\partial x \partial y}v_1v_2 + \frac{\partial^2 f_2}{\partial y^2}v_2^2 
    \end{pmatrix}
    = \frac{\partial^2 f_1}{\partial x^2} = 2r_1 \neq 0.
\end{align*}
Therefore, according to Lemma \ref{lem03}, a transcritical bifurcation occurs at $E_0$. Given the absence of stable interior equilibria or periodic solutions, and the saddle point nature of $E_0$, a closed manifold originating and terminating at $E_1$ exists.
\end{proof}

Figure \ref{figure8} illustrates a transcritical bifurcation in system \eqref{equ02} for parameter values $r_1 = 1.5$, $k_1 = 3$, $\lambda = 0.3$, $A = 0.8$, $b = 0.5$, $h = 0.7$, and $s = 0.5$. The left panel shows the phase portrait for $k_0 = 0.3$, while the right panel depicts the phase portrait for $k_0 = -0.3$. In the latter case, a stable limit cycle emerges.
\begin{figure}[H]
  \begin{center}  
    \includegraphics[scale=0.5]{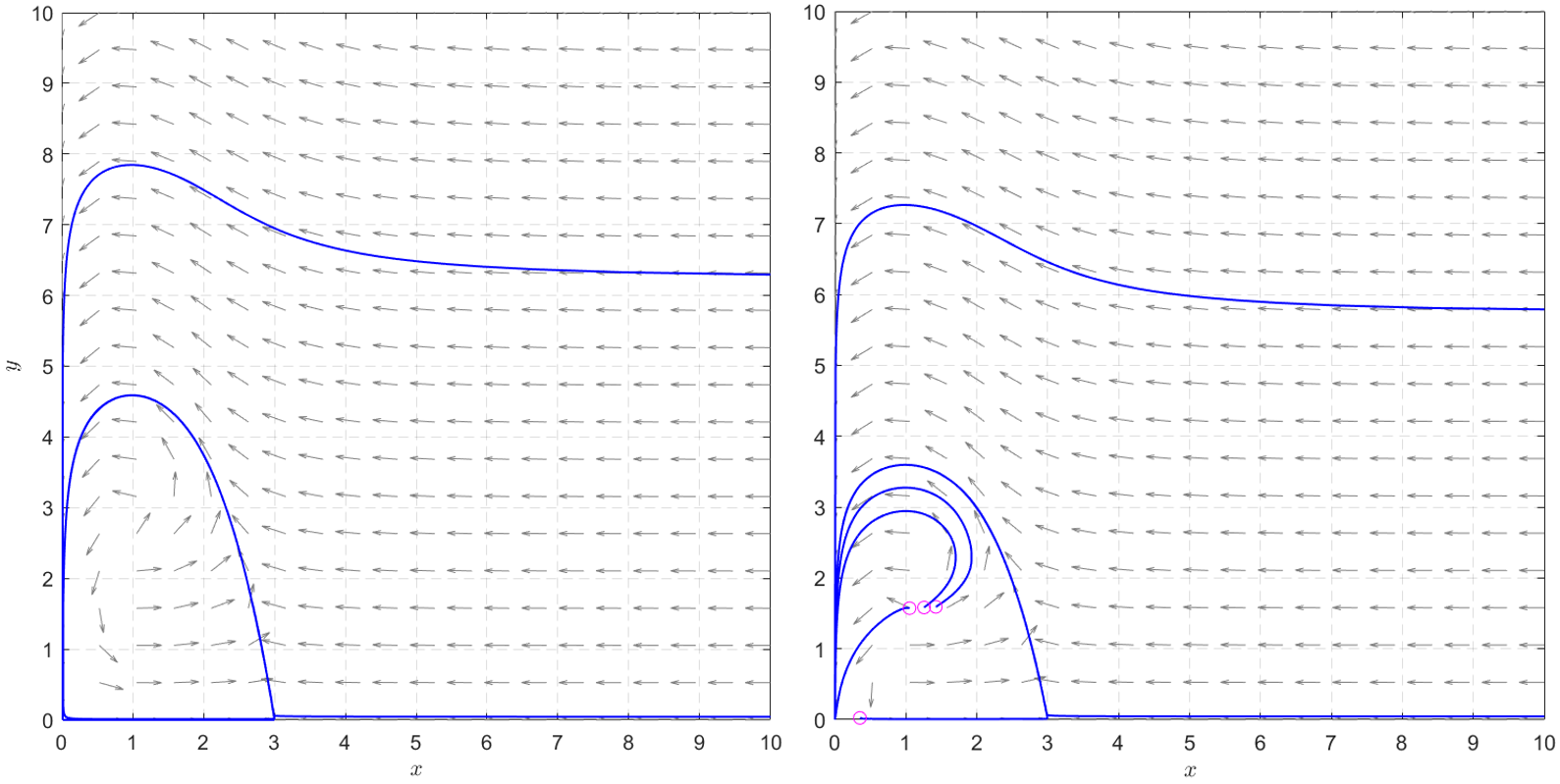}
    \caption{Transcritical bifurcation. Left panel: $k_0 = 0.3$.  Right panel: $k_0 = -0.3$. A stable limit cycle exists in the right panel.}  \label{figure8}
  \end{center}
\end{figure}

\section{Conclusions}\label{sec9}

This paper investigates a predator-prey model in which the prey exhibits the Allee effect and the predators engage in cooperative hunting, with an emphasis on the impact of random noise on population dynamics.

We begin by analyzing the deterministic system, where we establish the existence and uniqueness of local and global solutions. We identify conditions for solution boundedness and the existence of interior equilibrium points. Through Jacobian matrix analysis, we assess the stability of both boundary and interior equilibria and provide conditions for global solution stability. Additionally, we apply bifurcation theory to study transcritical bifurcations, saddle-node bifurcations, Hopf bifurcations. Notably, we prove the existence of heteroclinic orbits and establish conditions for both heteroclinic and transcritical bifurcations that have global significance.

This comprehensive analysis of the system provides valuable insights into the dynamics of predator-prey interactions under both deterministic conditions, with implications for ecological modeling and the management of such populations.


\end{document}